\documentclass[a4paper, 12pt]{article}
 \usepackage[english]{babel}
\usepackage{amssymb}
\usepackage{amsmath}
\usepackage{graphics}

\usepackage{amsthm, amssymb}
\usepackage{amsfonts}
\usepackage{epsfig,multicol}

\newtheorem{theorem}[subsection]{Theorem}
\newtheorem{proposition}[subsection]{Proposition}
\newtheorem{lemma}[subsection]{Lemma}
\newtheorem{corollary}[subsection]{Corollary}
\newtheorem{definition}[subsection]{Definition}

\newtheorem{remark}[subsection]{Remark}

\def\<{\langle}
\def\>{\rangle}

\begin{document}

\def\hpic #1 #2 {\mbox{$\begin{array}[c]{l} \epsfig{file=#1,height=#2}
\end{array}$}}
 
\def\vpic #1 #2 {\mbox{$\begin{array}[c]{l} \epsfig{file=#1,width=#2}
\end{array}$}}

\title{A quadrilateral in the Asaeda-Haagerup category}
\author{Marta Asaeda 
\thanks{
 Department of Mathematics;
University of California, Riverside;
900 University Ave;
Riverside, CA 92521;
USA. The first named author was partially supported by NSF grant DMS-0504199.} 
and Pinhas Grossman
\thanks{School of Mathematics;
Cardiff University;
Cardiff, CF24 4AG;
Wales, UK. The second named author was partially supported by NSF grant DMS-0801235 and by EU-NCG grant MRTN-CT-2006-031962.}
} 
 \maketitle 
\begin{abstract}
We construct a noncommuting quadrilateral of factors whose upper sides are each the Asaeda-Haagerup subfactor with index $\frac{5+\sqrt{17}}{2} $ by showing the existence of a Q-system in the Asaeda-Haagerup category with index $\frac{7+\sqrt{17}}{2} $. 
  \end{abstract}

\section{Introduction}
Subfactor theory was initiated by Jones as a noncommutative Galois theory \cite{J3}. It is therefore natural to study the lattice of intermediate subfactors of a finite-index subfactor as a quantum analogue of the subgroup lattice of a finite group. The problem of classifying lattices of intermediate subfactors was posed by Watatani \cite{Wat3}, and recent progress has been made by Xu \cite{X19}. 

The simplest nontrivial lattice is a single proper intermediate subfactor $N \subset P \subset M$. Such inclusions were studied by Bisch and Jones, and they provided a generic construction in terms of the index parameters $[M:P]$ and $[P:N]$ \cite{BJ}. Their construction was a free composition, in the sense that the $P-P$ bimodules coming from $N \subset P$ and $P \subset M$ have free relations; their results show that there is no obstruction in this case.

The next simplest case is a pair of distinct intermediate subfactors: $$\begin{array}{ccc}
P&\subset &M \cr
\cup& &\cup \cr
N&\subset &Q
\end{array}$$ Such a configuration is called a quadrilateral of factors if $P \vee Q =M$ and $P \wedge Q = N$.  The presence of an additonal intermediate subfactor means that we are no longer in a free situation. An important notion is commutativity, which means that the trace-preserving conditional expectations of $M$ onto $P$ and $Q$ commute. There is also the dual notion of cocommutativity. Commuting, cocommuting quadrilaterals may be constructed via a tensor product, but it turns out that noncommutativity imposes a great deal of rigidity. Sano and Watatani studied noncommuting quadrilaterals of factors and introduced the notion of angles between subfactors, a numerical invariant which measures the noncommutativity \cite{SaW}.

In \cite{GrJ}, the second named author and Jones studied noncommuting quadrilaterals of factors whose sides are supertransitive, a minimality condition which means that the planar algebras are generated by Temperley-Lieb diagrams. They found that there are only two examples of such quadrilaterals up to isomorphism of the planar algebra, a cocommuting quadrilateral coming from an outer action of $S_3$ on a factor, with $2=[M:P]=[P:N]-1$ and a noncocommuting quadrilateral with $[M:P]=[P:M]=2+\sqrt{2} $. In \cite{GI}, the second named author and Izumi showed that if the sides are only required to be $3$-supertransitive, then one still has $[M:P]=[P:N] $ for noncocommuting quadrilaterals and $[M:P]=[P:N]-1 $ for cocommuting quadrilaterals (in fact all that is required is that the sides are $2$-supertransitive and $N \subset P $ has trivial second cohomology in the sense of Izumi and Kosaki \cite{IKo3} ). In the latter case, one has the Galois group $Gal(M/N) \subseteq S_3$, with equality only for the fixed point subfactor of an outer action of $S_3$ on a factor. Moreover, if $\{e\} \subset Gal(M/N) \subset S_3$ then one has the following relation among the $P-P$ bimodules of the quadrilateral: ${}_P P_N \otimes_N {}_N P_P \cong {}_P P_P \oplus ( {}_P M_M \otimes_M {}_{\alpha(M)} M_M \otimes_M {}_M M_P $), where $\alpha$ is an outer automorphism of $M$. In sector notation, this relation is $[\iota \bar{\iota} ] = [Id_P ] \oplus [\bar{\kappa} \alpha \kappa ]$, where $\iota = {}_P P_N$ and $\kappa = {}_M M_P $.

Subfactors with index less than $4$ must have index $4cos^2 \frac{\pi}{k} $ by Jones' index theorem \cite{J3}. The principal graphs were classified by Ocneanu as type $A_n$, $D_{2n}$, $E_6$, and $E_8$ Dynkin diagrams. Note that these are all finite graphs, a condition called finite depth. Principal graphs of subfactors with index $4$ have been classified as certain extended Dynkin diagrams; some of these are infinite \cite{P12}. In \cite{AH}, the first named author and Haagerup constructed two exotic finite-depth subfactors with indices $\frac{5+\sqrt{13}}{2} $ (known as the Haagerup subfactor) and $\frac{5+\sqrt{17}}{2} $ (known as the Asaeda-Haagerup subfactor). Along with the recently constructed \cite{BMPS} extended Haagerup subfactor, these (and their duals) are the only finite-depth subfactors with indices strictly between $ 4$ and $3 + \sqrt{3}$ \cite{H5}.  

In \cite{GI}, all noncommuting, irreducible quadrilaterals with sides of index less than or equal to four were classified, up to isomorphism of the planar algebra; there are seven such quadrilaterals. Moreover, it was shown that the Haagerup subfactor appears as the upper sides of both types of quadrilaterals: there is a noncommuting, noncocommuting quadrilateral all of whose sides ae Haagerup subfactors; and there is also a noncommuting, but cocommuting quadrilateral whose upper sides are the Haagerup subfactor but whose lower sides have index $\frac{7+\sqrt{13}}{2} $. This quadrilateral has Galois group $\mathbb{Z} / 3\mathbb{Z} $, and is in fact the only known example of a noncommuting quadrilateral with $2$-superransitive sides and this Galois group.

There was considerable evidence that the Asaeda-Haagerup subfactor should appear in a quadrilateral as well. While it cannot appear in a noncocommuting quadrilateral, it was in fact shown in \cite{GI} that any noncommuting but cocommuting quadrilateral which has $\mathbb{Z} / 2\mathbb{Z} $ Galois group and is maximally supertransitive, in the sense that the upper sides are $5$-supertransitive and the lower sides are $3$-supertransitive, must have upper sides with principal graph containing the Asaeda-Haagerup graph. Moreover, the candidate prinipal graph for the lower sides of such an Asaeda-Haagerup quadrilateral, along with two other graphs of the same index, were found independently by Morrison, Peters, and Snyder while searching for possible prinipal graphs which start off as the Haagerup graph.

The proof of the existence of the cocommuting Haagerup quadrilateral involved showing the existence of a Q-system for $[Id_P ] \oplus [\bar{\kappa} \alpha \kappa ]$, where $\kappa \bar{\kappa }$ is a Q-system for a Haagerup subfactor, and $\alpha $ is the period $3$ automorphism corresponding to the symmetry in the Haagerup graph. The quadrilateral then is obtained by composing these two Q-systems. The main technical difficulty was in verifying the Q-system relations. This was accomplished through heavy use of the diagrammatic calculus for tensor categories; the diagrams were ultimately evaluated in terms of generators of a Cuntz algebra, using Izumi's construction of the Haagerup subfactor from endomorphsims of a Cuntz algebra \cite{I10}.

The problem with doing the same thing in the Asaeda-Haagerup category is that there is no corresponding Cuntz algebra representation, so it is difficult to evaluate intertwiner diagrams explicitly. However, there is one principal advantage of the Asaeda-Haagerup category over the Haagerup category: since the graph automorphism has period $2$ instead of period $3$, the intertwiner equations that occur in the Q-system relations are all in $1$ dimensional spaces, i.e. they are essentially scalar equations. This allows us to verify the equations by comparing nonzero ``states'' of the diagrams, rather than fully computing the whole diagrams. The formalism used to express and evaluate these states is very similar to Jones' bipartite graph planar algebra formalism \cite{J21}.

Once the existence of the ``plus one'' subfactor is established, its principal graph may be easily computed. The dual graph was given to the authors by Noah Snyder using the subfactor atlas \\ ( http://tqft.net/wiki/Atlas\_of\_subfactors). Interestingly, this dual graph has a symmetry very similar to the original Asaeda-Haagerup graph, leading us to conjecture that the construction may be iterated once more: i.e. there may exist a subfactor in the Asaeda-Haagerup category with index $\frac{9+\sqrt{17}}{2} $ and associated quadrilteral with upper sides having the new Asaeda-Haagerup ``plus one'' graphs. Checking this conjecture should be straightforward using the methods of this paper combined with methods of \cite{AH}, but requires some computation. We hope to do this soon. 

Aside from the application to classification of quadrilaterals, the existence of the AH+1 subfactor should be of independent interest as there is a dearth of finite depth subfactors with small index.

The paper is organized as follows: after the present introductory section, Section 2 is a background section reviewing some basic facts about Q-systems and biunitary connections. Section 3 proves some identities of intertwiners in the Asaeda-Haagerup category; this is a lot of the workload of the proof of the main theorem. In Section 4 we prove the existence of the Asaeda-Haagerup plus one subfactor and the associatd quadrilateral.

\textbf{Acknowledgements.}
The authors would like to thank Masaki Izumi for conjecturing the existence of the Asaeda-Haagerup quadrilateral, which along with his construction of the Haagerup quadrilateral is the inspiration for the present work; and for many helpful conversations. The authors would like to thank Noah Snyder for helpful comments on the manuscript and for finding the dual graph of the new Asaeda-Haagerup ``plus one'' subfactor with the  subfactor atlas (http://tqft.net/wiki/Atlas\_of\_subfactors).

\section{Preliminaries}

\subsection{Subfactors, bimodules, and Q-systems}
Let $M$ be a Type II$_1$ factor with unique normalized trace
$tr$, and let $1 \in N \subset M$ be a finite-index subfactor. Let $\kappa$ and
$\bar{\kappa}$ denote, respectively, the Hilbert space completions of the
multiplication bimodules ${}_N M {}_M $ and ${}_M M {}_N $ with respect to $tr$.

Following sector notation, we will often omit the tensor symbol when writing
relative tensor products, so that e.g. $\kappa \bar{\kappa} $ means
$\kappa \otimes_M \bar{\kappa}$. For any two $A-B$ bimdules $\rho $ and
$\sigma$, the intertwiner space $Hom_{A,B}(\rho, \sigma ) $ will be denoted by
$(\rho, \sigma )$. We have two
distinguished bimodules $Id_N := {}_N L^2(N) {}_N $ and $Id_M := {}_M L^2(M) 
{}_M$.
 Finally, if $\rho $ is an $A-B$ bimodule, $\sigma $ is a $B-C$ bimodule, and
$\lambda $ is a $C-D$ bimodule, then the bimodules $(\rho \sigma ) \lambda $
 and $\rho (\sigma \lambda ) $ are naturally isomorphic, and we will think of
them as being identified via this isomorphism. Similarly, $\rho, \rho \otimes
Id_B $ and $Id_A \otimes \rho $ are naturally isomorphic and we will identify
these as well.

We recall Longo's conjugacy theory \cite{L2}, which was originally
formulated for endomorphisms of Type III factors and translated to the finite
setting by Masuda \cite{M1}. There exist isometries $r_{\kappa} \in ( Id_N, \kappa
\bar{\kappa}) $ and
$\bar{r}_{\kappa} \in (Id_M, \kappa \bar{\kappa} ) $ satisfying 
\begin{equation} \label{con1}
(r_{\kappa}^*
\otimes Id_{\kappa}) \circ (Id_{\kappa} \otimes \bar{r}_{\kappa}) = \frac{1}{d}
Id_{\kappa} 
\end{equation}

 \begin{equation} \label{con2}
  (Id_{\bar{\kappa}} \otimes
r_{\kappa}^*)\circ(\bar{r}_{\kappa} \otimes Id_{\bar{\kappa}})=\frac{1}{d}
Id_{\bar{\kappa}} 
 \end{equation}
where $d=[M:N]$ is the Jones index of $N \subset M$.

We will make heavy use of the diagrammatic calculus for tensor categories, in which morphisms are represented by vertices from which emanate strings labeled by the origin objects (upwards) and by the destination objects (downwards). Straight strings labeled by objects correspond to identity morphisms, and strings labeled by identity objects are often suppressed. Tensoring is depicted by horizontal concatenation, and composition by vertical concatenation. Diagrams are read from top to bottom.

Then if we let $\hpic{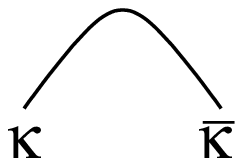} {0.5in} = \sqrt{d} r_{\kappa}$ and $\hpic{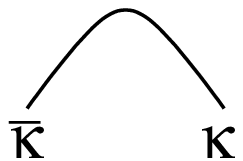} {0.5in} = \sqrt{d} \bar{r}_{\kappa}$, the above equations become: $\hpic{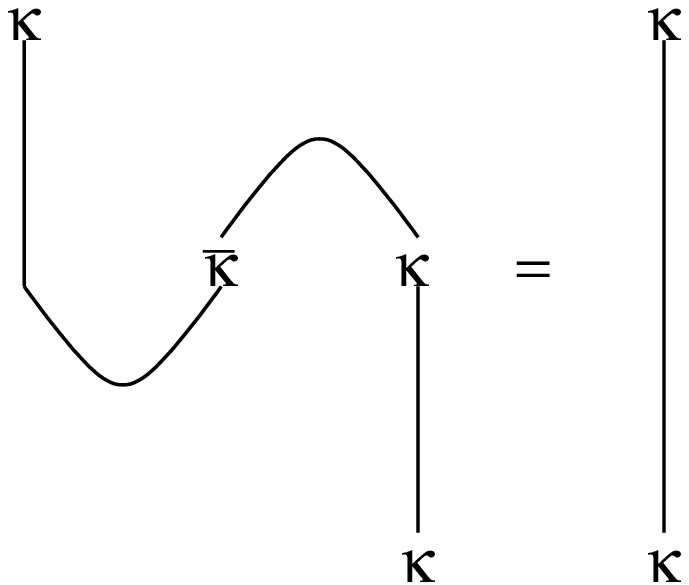} {1.0in} $ and $\hpic{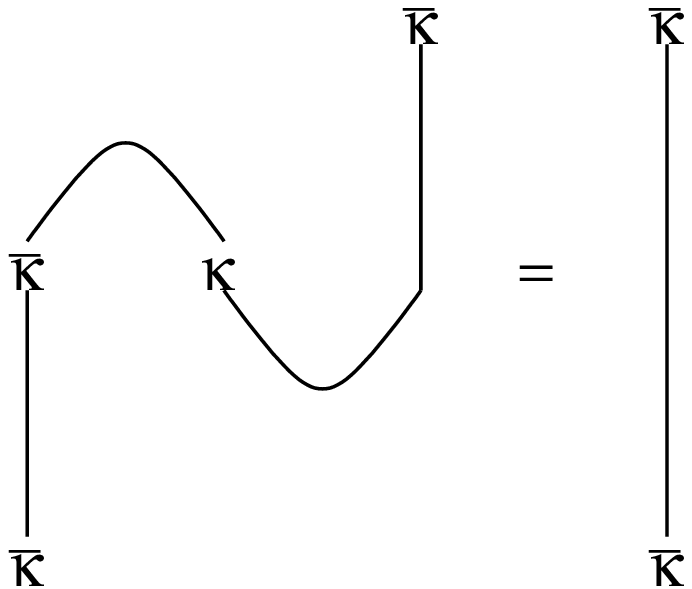} {1.0in} $.

\begin{definition}
A Q-system over a II$_1$ factor $N$ is a triple $(\gamma,T,S)$ where $\gamma$ is
an
$N-N$ bimodule with $dim({}_N \gamma ) = dim(\gamma_N ) $, $T \in (Id_N,
\gamma)$ and $S \in (\gamma,\gamma^2 )$ are isometries, and such that:\\
(1)  $(S \otimes Id_{\gamma}) \circ S = (Id_{\gamma} \otimes S) \circ S$\\
(2) $(T^* \otimes Id_{\gamma} ) \circ S = ( Id_{\gamma} \otimes T^*) \circ S =
\frac{1}{d} Id_{\gamma}$ for some $d > 0$.
\end{definition}

Note that while the definition in \cite{M1} included the additional condition $ SS^* = (S \otimes Id_{\gamma} ) \circ (Id_{\gamma} \otimes S^*)$, this condition was shown to be redundant in \cite{LRo}. 

\begin{theorem} \cite{L2}, \cite{M1} If $N \subset M$ is a II$_1$ subfactor, then $(\kappa \bar{\kappa},
r_{\kappa }, Id_{\kappa} \otimes \bar{r}_{\kappa} \otimes Id_{\bar{\kappa}}
) $ is a Q-system. Conversely, any Q-system over $N$ arises in this way for some
$M \supset N$. 
 \end{theorem}

If $\gamma \cong Id_N \oplus \sigma $ where $\sigma$ is irreducible, the
Q-system equations can be simplified; the following result was stated in \cite{GI} for infinite factors but is equally true for Type II$_1$ factors:

\begin{proposition} \label{qsystem_1plus}
 Let $\sigma$ be a self-conjugate $N-N$ bimodule such that $dim({}_N \sigma ) =
dim(\sigma {}_N ) = d $ and $\sigma \ncong Id_N$. Then $Id_N \oplus \sigma $ admits a Q-system iff there
exist isometries $R \in (Id_N, \sigma^2)$ and $S \in (\sigma, \sigma^2 ) $ such
that\\
(1) $(S \otimes Id_{\sigma} ) \circ R = ( Id_{\sigma} \otimes S) \circ R $\\
(2) $\displaystyle \frac{\sqrt{d+1}}{d} (R \otimes Id_{\sigma}- Id_{\sigma}
\otimes R) = (Id_{\sigma} \otimes S) \circ S - (S \otimes Id_{\sigma}) \circ S
$.

\end{proposition}

In pictures, if we set $ \hpic{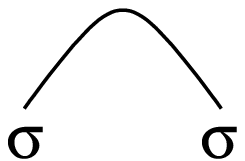} {0.3in} = \sqrt{d} R  $ and $\hpic{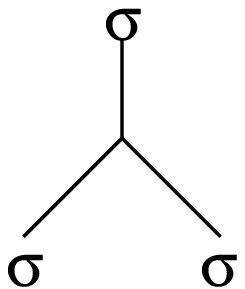} {0.5in} = d^{\frac{1}{4}} S$, then this becomes:\\\\

(1) $\hpic{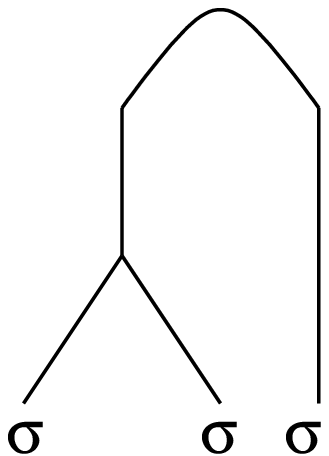} {1.0in} = \hpic{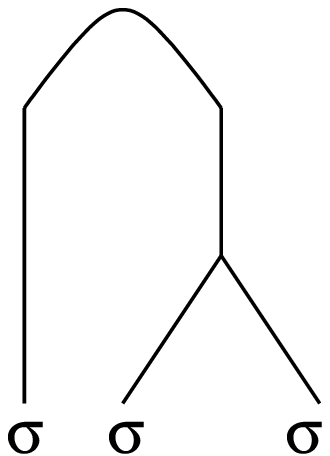} {1.0in} $\\\\ 

(2) $\displaystyle \frac{\sqrt{d+1}}{d} (\hpic{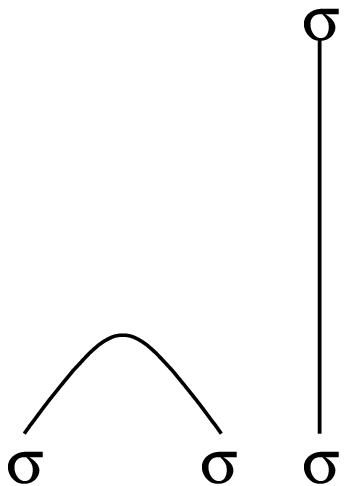} {1.0in} - \hpic{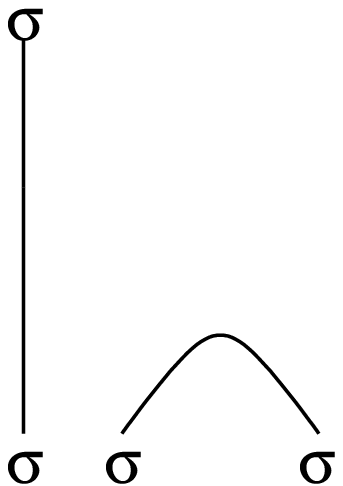} {1.0in} ) =  \hpic{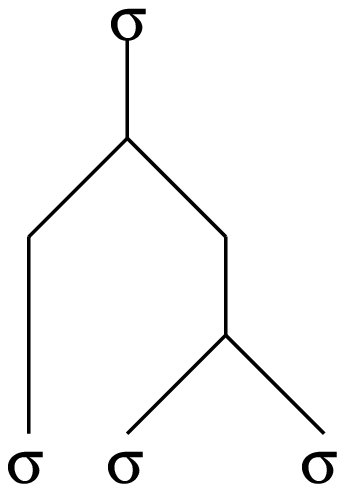} {1.0in} - \hpic{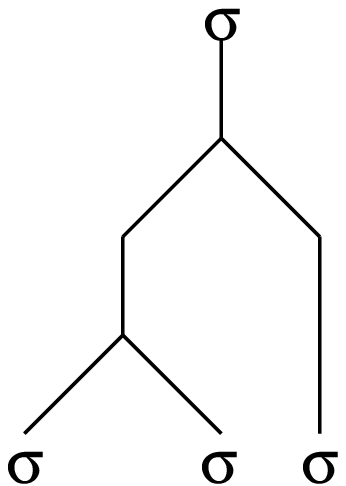} {1.0in} $.

\subsection{Connections and bimodules}

Let $N \subset M$ be a finite index subfactor. The even vertices in the principal (respectively, dual) graph corresponded to the irreducible $N-N$
($M-M$) bimodules which occur in the decomposition of the tensor powers of
$\kappa \bar{\kappa} $ ($\bar{\kappa} \kappa$), and the odd vertices to the
$N-M$ ($M-N$) bimodules which occur in the decomposition of the even bimodules
tensored again on the right by $\kappa$ ($\bar{\kappa}$), where $\kappa $ as before is the completion of ${}_N
M _M $.

To handle bimodules concretely, we use Ocneanu's paragroup theory. If $N \subset M $ has finite depth, the $N-N$ (resp. $N-M$) bimodules may be
represented as biunitary connections whose horizontal graphs are both the
principal graph (resp. whose upper graph is the principal graph and whose lower
graph is the dual graph) of $N \subset M$. 

\def\gb{{\beta^2}}
\def\gbf{{\beta^4}}
\def\gbs{{\beta}}
\def\os{{\rm OpString}}
\def\cho{\choose}
\def\ovl{\overline}
\def\sk{{*_{\cal K}}}
\def\sl{{*_{\cal L}}}
\def\sg{{*_{\cal G}}}
\def\lan{\langle}
\def\ran{\rangle}
\def\oX{{\ovl X}}
\def\begeq{\begin{eqnarray*}}
\def\endeq{\end{eqnarray*}}
\def\pha{\phantom}
\def\la{\lambda}
\def\lad{\lambda^2}
\def\R{R}
\def\Y{Y}
\def\Z{Z}
\def\sm{\small}
\def\s#1{$#1_{\sigma}$}
\def\ss#1{$#1_{\sigma^2}$}
\def\si#1{#1_{\sigma}}
\def\ssi#1{#1_{\sigma^2}}
\def\a{\alpha}
\def\ab{\tilde{\alpha}}
\def\dline(#1,#2){\multiput(#1,#2)(0,-3){8}{\line(0,-1){2}}}
\def\hl{\hline}
\def\mcol{\multicolumn}
\def\td#1{\tilde #1}
\def\bl{$\bullet$}
\begin{figure}[h]
\begin{center}
\thinlines
\unitlength 1.0mm
\begin{picture}(30,20)(0,-5)
\multiput(11,-2)(0,10){2}{\line(1,0){8}}
\multiput(10,7)(10,0){2}{\line(0,-1){8}}
\multiput(10,-2)(10,0){2}{\circle*{1}}
\multiput(10,8)(10,0){2}{\circle*{1}}
\put(6,12){\makebox(0,0){$V_0$}}
\put(24,12){\makebox(0,0){$V_1$}}
\put(6,-6){\makebox(0,0){$V_3$}}
\put(24,-6){\makebox(0,0){$V_2$}}
\put(15,12){\makebox(0,0){${\cal G}_0$}}
\put(15,-6){\makebox(0,0){${\cal G}_2$}}
\put(6,3){\makebox(0,0){${\cal G}_3$}}
\put(24,3){\makebox(0,0){${\cal G}_1$}}
\put(15,3){\makebox(0,0){$\a$}}
\end{picture}
\end{center}
\caption{Schematic representation of a connection; cells are based loops around the four graphs.  }
\end{figure}

In this formalism, direct sums of
bimodules are then given by merging the vertical graphs as a disjoint union, and tensor
products are given by composing the vertical graphs, and the connection
accordingly. For more details, we refer the reader to \cite{EK7} and \cite{AH}.

The connection may be extended linearly to cells composed of formal linear
combinations of edges, i.e. elements of the Hilbert spaces associated to each
pair of vertices with orthonormal basis indexed by the edges between those
vertices. 

\begin{definition}
 A vertical gauge transformation between two biunitary connections on the same
graphs is a unitary map on the vertical edge spaces
which commutes with the connections. 
\end{definition}

We recall the following result from \cite{AH}.

\begin{theorem}
 A vertical gauge transformation gives an isomorphism between the associated
bimodules. Conversely, if two bimodules given by biunitary connections with the
same
horizontal grahps are isomorphic, then the vertical graphs are also identical
and the bimodule isomorphism is
given by a vertical gauge transformation.
\end{theorem}

\begin{definition}
 Let $\rho$ and $\sigma$ be bifinite $A-B$ bimodules represented by biunitary connections with the same horizontal graphs but possibly different vertical graphs. A vertical sub-gauge transformation is a collection of partial isometries on the vertical edge spaces which commutes with the connections when restricted to the orthogonal complements of the kernels.  
\end{definition}

\begin{corollary}
Any partial isometry between $\rho$ and $\sigma$ is given by a vertical sub-gauge transformtion. Any intertwiner between $\rho $ and $\sigma$ can be expressed as a linear combination of vertical sub-gauge transformations.  
\end{corollary}

\begin{proof}
 Since $\rho$ and $\sigma$ are bifinite, any partial isometry can be written as an isomorphism (between possibly larger bimodules) multiplied by an orthogonal projection. Since the isomorphism is given by a vertical gauge transformation, multiplying by an orthogonal projection corresponds to taking a vertical sub-gauge transformation. For the second part, note that the intertwiner space $(\rho,\sigma)$ is spanned by partial isometries.
\end{proof}

Intertwiners compose the same way as maps on the vertical edge spaces, and act linearly componentwise on composite edges in tensor products. By a slight abuse of notation, we will often identify bimodules with their associated connections in the sequel.
%

\section{Intertwiners in the Asaeda-Haagerup category}

In \cite{AH}, the first-named author and Haagerup constructed a subfactor $N \subset M$ with index $\frac{5+\sqrt{17}}{2} $.
The principal graph is $$\hpic{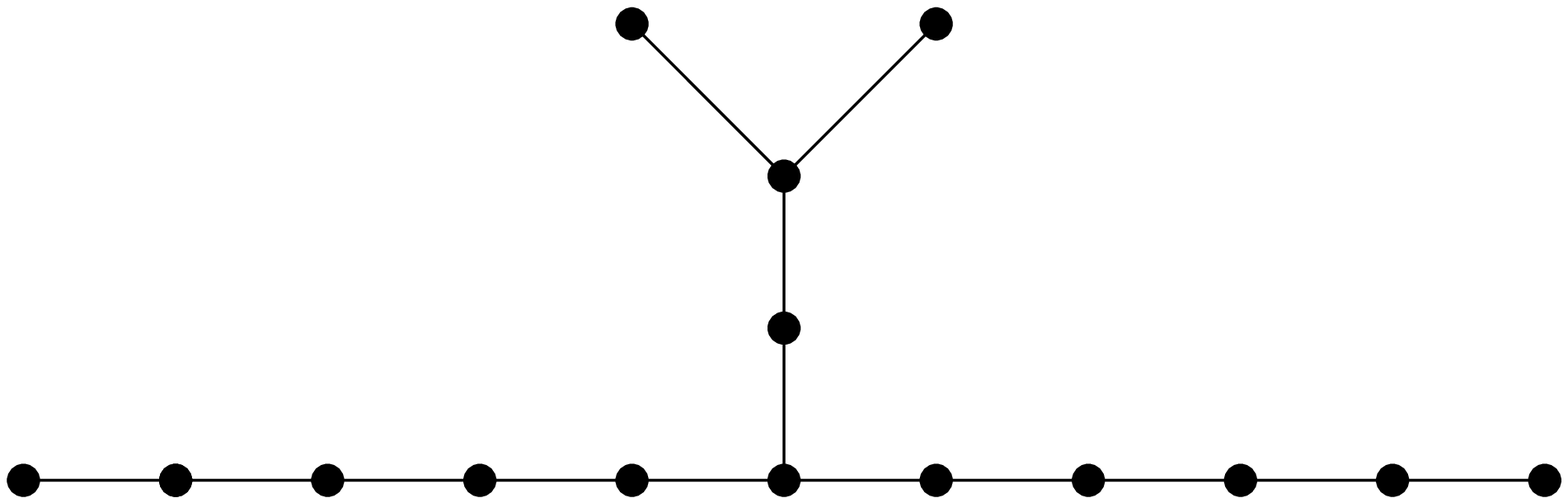} {0.7in} $$ and the dual graph is $$\hpic{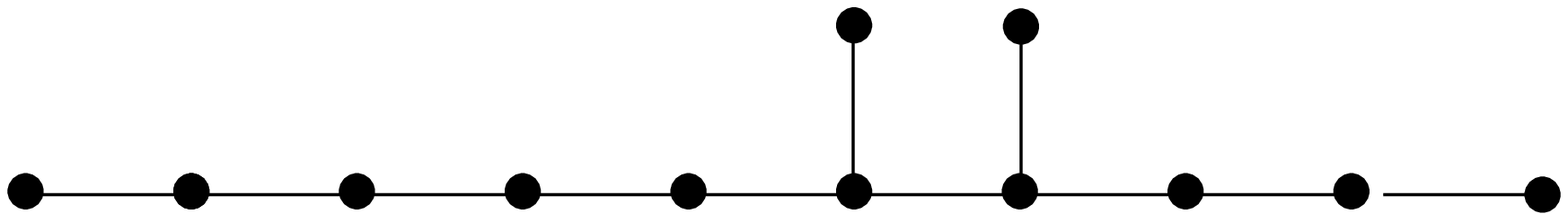}
{0.3in} .$$ They explicitly computed the biunitary connection for $\kappa = {}_N L^2(M)_M$. For future reference we include the four graphs of $\kappa$:

$$\hpic{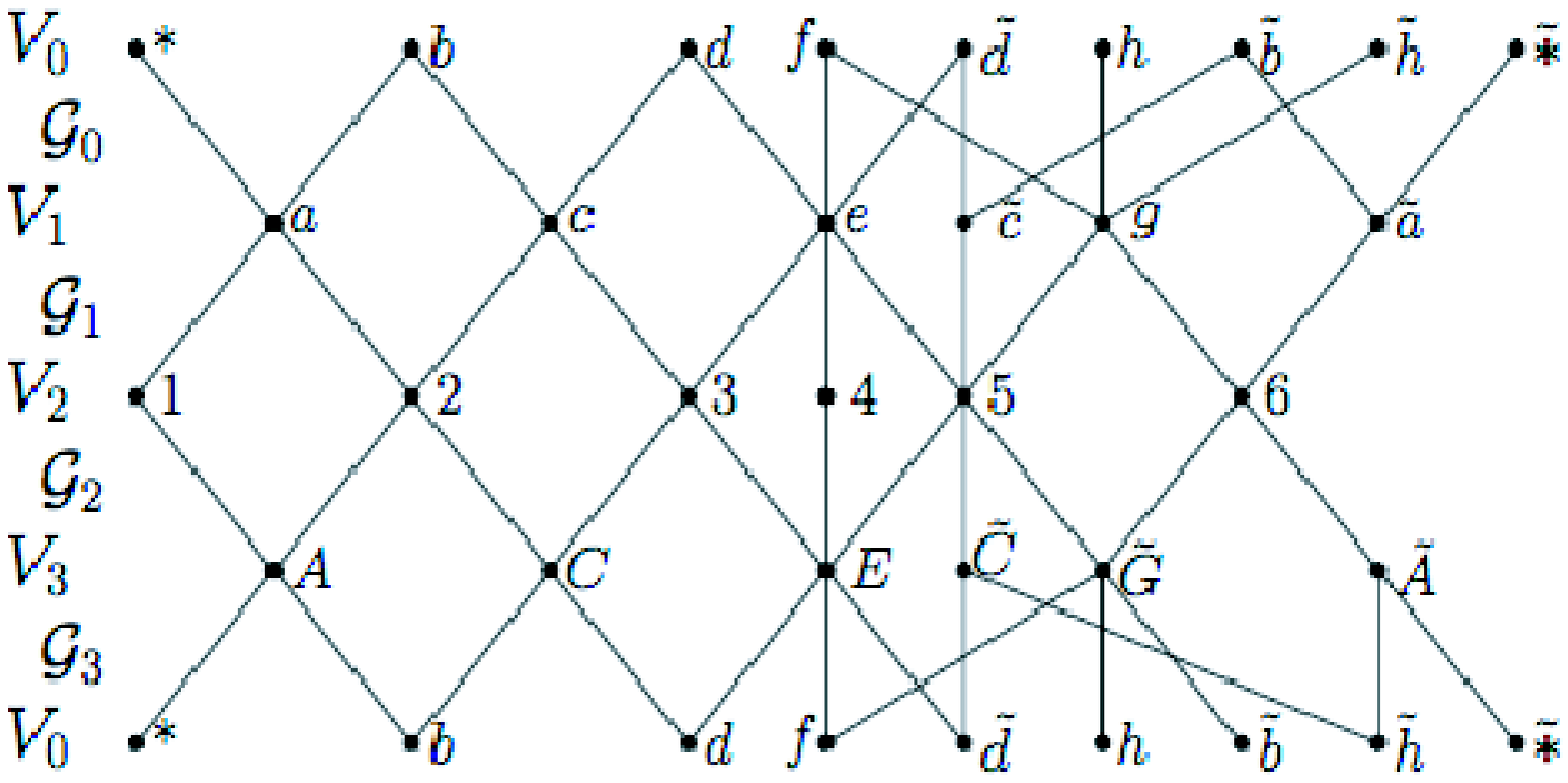} {2.0 in} $$

We imagine these four graphs as being wrapped in a square, so that the first and third graphs from the top are the ``horizontal'' graphs and the second and fourth are the ``vertical'' graphs in the square. 
\begin{figure}[h]
\begin{center}
\thinlines
\unitlength 1.0mm
\begin{picture}(30,20)(0,-5)
\multiput(11,-2)(0,10){2}{\line(1,0){8}}
\multiput(10,7)(10,0){2}{\line(0,-1){8}}
\multiput(10,-2)(10,0){2}{\circle*{1}}
\multiput(10,8)(10,0){2}{\circle*{1}}
\put(6,12){\makebox(0,0){$V_0$}}
\put(24,12){\makebox(0,0){$V_1$}}
\put(6,-6){\makebox(0,0){$V_3$}}
\put(24,-6){\makebox(0,0){$V_2$}}
\put(15,12){\makebox(0,0){${\cal G}_0$}}
\put(15,-6){\makebox(0,0){${\cal G}_2$}}
\put(6,3){\makebox(0,0){${\cal G}_3$}}
\put(24,3){\makebox(0,0){${\cal G}_1$}}
\put(15,3){\makebox(0,0){$\a$}}
\end{picture}
\end{center}
\caption{Follow clockwise from top left to get the vertical picture}
\end{figure}
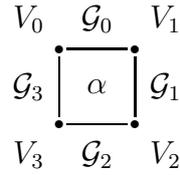

Then composing with the dual gives the graphs of $\kappa \bar{\kappa} $ (Note that the left vertical graph $G_3$ is ``upside down'', so we reflect vertically before composing):

$$\hpic{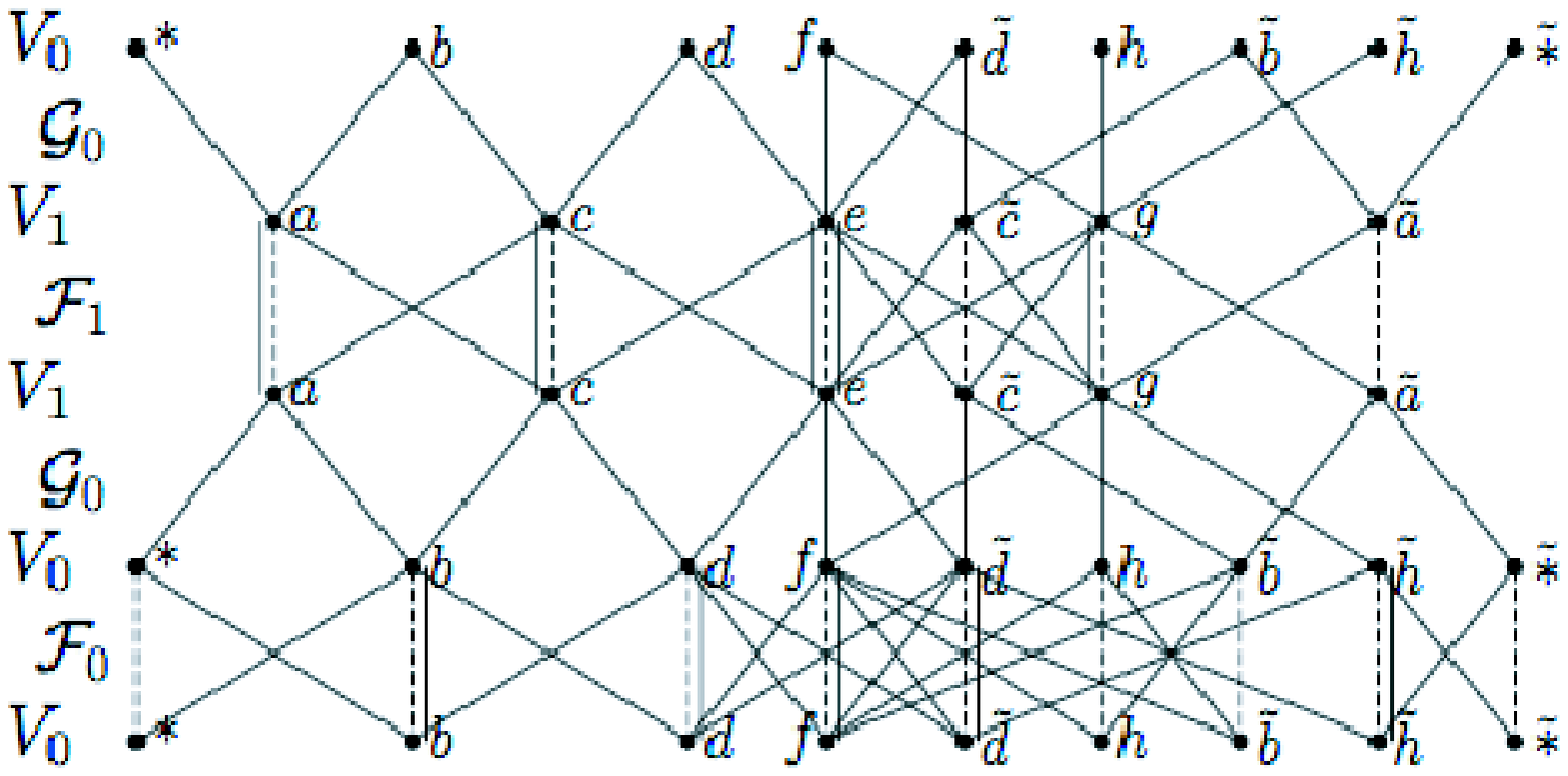} {2.0in} $$

The connection $\kappa \bar{\kappa} $ decomposes into $Id_N$, whose vertical graphs are indicated by the dotted lines, and whose value on every cell is $1$, and another connection, which we will call $\rho$. The vertical graphs for $\rho$ are given by the complements of the dotted lines in the graphs of $\kappa  \bar{\kappa}$; the connection is determined only up to vertical gauge choice and a representative, which we will take as well, was computed in \cite{AH}. We will call the connection corresponding to the vertex symmetric with respect to $Id_N$ in the principal graph $\alpha$; it is a $1$-dimensional $N-N$ bimodule. The  vertical graphs of $\alpha$ switch $x$ with $\tilde{x}$ for each $x$ (we take $\tilde{\tilde{x}}=x$ for all $x$ and $\tilde{y}=y$ for those vertices $y$ which do not have any labeled ``$\tilde{y}$ '').

The principal graph can then be labeled by bimodules as follows:
\begin{center}
\hpic{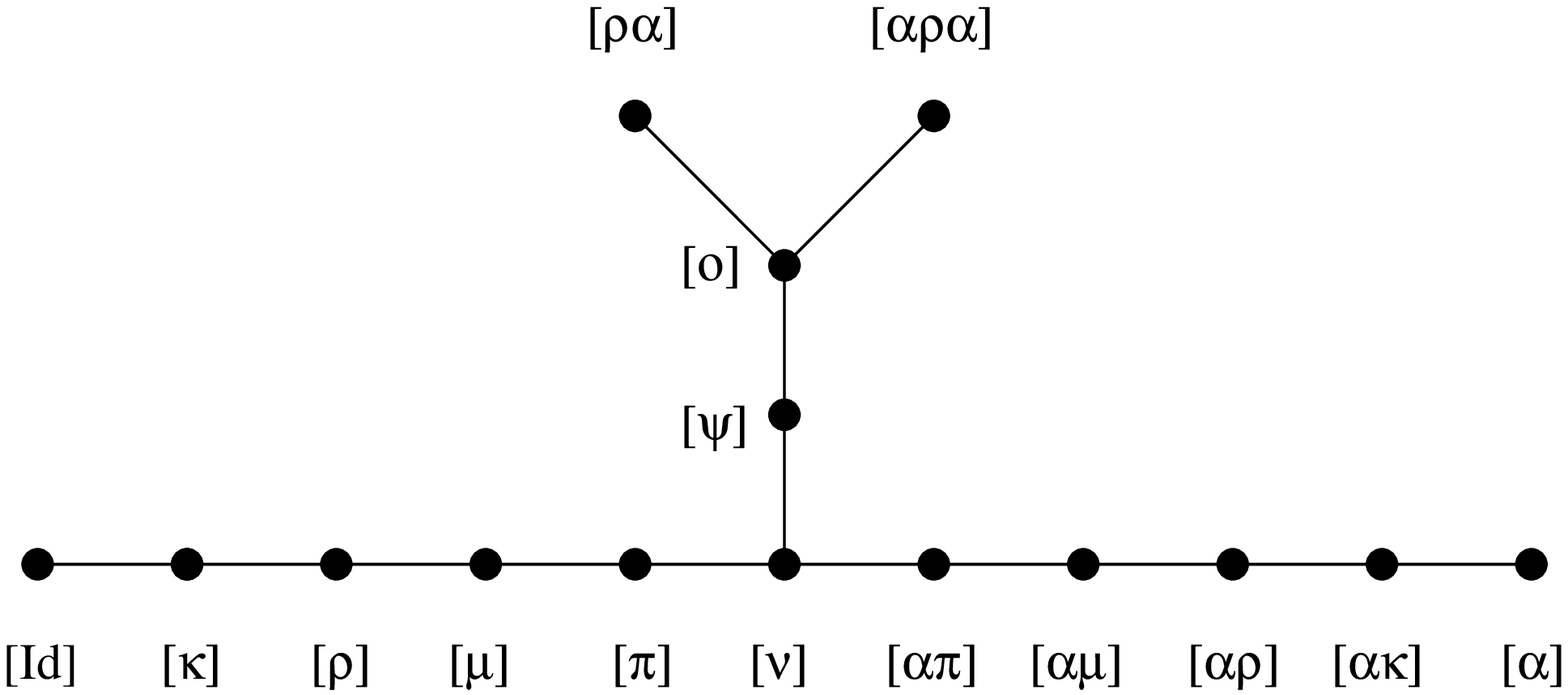} {1.2in} 
\end{center}

Following sector notation, the square brackets here denote isomorphism classes. The fact that $[\alpha \rho] \neq [\rho \alpha ] $ implies that the $M-M$ bimodule $\bar{\kappa} \alpha \kappa $ is irreducible: we have $dim(\bar{\kappa} \alpha \kappa, \bar{\kappa}  \alpha \kappa) = \dim( \alpha \kappa \bar{\kappa}, \alpha \kappa \bar{\kappa}   )= dim(Id_N \oplus \alpha \rho, Id_N \oplus \rho \alpha ) = 1$. Also note the fusion rule $[\rho \alpha \rho] = [\alpha \rho \alpha] \oplus [\eta] $. This implies that $dim(\bar{\kappa} \alpha \kappa,(\bar{\kappa} \alpha \kappa)^2)=1$.

We want to fix certain intertwiners. Recall that an intertwiner between two connections with the same horizontal graphs is given by a collection of maps on the vertical edge spaces. To describe such an intertwiner, we list the maps corresponding to each edge in the vertical graphs of the origin of the intertwiner.

 We use the following notation. Each edge
will be denoted by the pair of vertices it connects, e.g. ``$*A$'' denotes the edge in the left vertical graph of $\kappa $ which connects $*$ to $A$. In principle there could be multiple such edges but in our computations all the edges will be simple. Composite edges
will be denoted by the vertices from each component, e.g. ``$*Ab$'' denotes the edge in the left vertical graph of $\kappa \bar{\kappa} $ which is composed of the edges $*A$ in $\kappa $ and  $Ab$ in  $\bar{\kappa}$. Finally, each gauge map will be represented as an edge of the origin
mapping to a linear combination of edges in the destination.

The intertwiners $r_{\kappa} $ and $\bar{r}_{\kappa}$ are determined up to a scalar. We can fix the scalar by assigning a complex unitary as the gauge entry corresponding to any simple edge in $\kappa \bar{\kappa}  $. So we fix $** \mapsto ** $ and then $r_{\kappa} $ can be computed using the methods of \cite{AH} . We also use the following notation of \cite{AH}: $$\beta=\sqrt{\frac{5+ \sqrt{17}}{2}}, \beta_n = \sqrt{\beta^2 - n}, \beta_n'=\beta_n^2, \gamma=\sqrt{2\beta^2-1}, \gamma'=\gamma^2 $$

Then the map $r_k$ is defined by the following table:\\\\

\begin{tabular*}{0.75\textwidth}{@{\extracolsep{\fill}} | l | l | }
\hline
 $** \mapsto ** $ & $aa \mapsto \frac{1}{\beta} a1a + \frac{\beta_2}{\beta} a2a$ \\
 $bb \mapsto \frac{1}{\beta_1} bAb + \frac{\beta_2}{\beta_1} bCb$ & $cc \mapsto \frac{\beta_1}{\beta \beta_2} c2c +
\frac{\beta_1}{\sqrt{2}\beta_2} c3c $\\
$dd \mapsto \frac{2}{\beta_1^2} dCd + \frac{2 \sqrt{2}}{\beta \beta_2} dEd $ & $ee \mapsto \frac{\beta_2}{2 \sqrt{2}} e3e + \frac{1}{\beta} e4e + \frac{\beta_2}{\beta_{-1}} e5e $\\
$ff \mapsto (\frac{1}{2}+ \frac{\sqrt{2}}{\beta \beta_{-1}} ) fEf + \frac{1}{\beta_2} fGf $ & $\tilde{c} \tilde{c} \mapsto \tilde{c}5\tilde{c}$\\
$\tilde{d} \tilde{d} \mapsto \frac{2 \sqrt{2}}{\beta \beta_2} \tilde{d}E \tilde{d} + \frac{2}{\beta_1^2} \tilde{d} \tilde{C} \tilde{d} $ & $gg \mapsto \frac{\beta_2}{\beta_1}g5g+ \frac{1}{\beta}g6g $\\
$hh \mapsto hGh$ & $\tilde{a} \tilde{a} \mapsto \tilde{a}G \tilde{a}$\\
$\tilde{b} \tilde{b} \mapsto \tilde{b}G \tilde{b}$ & \\
$\tilde{h} \tilde{h} \mapsto \frac{\beta_2}{\beta_1}\tilde{h}\tilde{C}\tilde{h} + \frac{1}{\beta_1}\tilde{h} \tilde{A}\tilde{h}$ & \\
$\tilde{*}\tilde{*} \mapsto \tilde{*}\tilde{A} \tilde{*}$ & \\
\hline
\end{tabular*}

\vskip2ex

Similarly, $\bar{r}_{\kappa} $ can be fixed by setting $11 \mapsto 1a1 $, and then we get the following table for  $\bar{r}_{\kappa} $:\\\\

\begin{tabular*}{0.75\textwidth}{@{\extracolsep{\fill}} | l | l | }
\hline
$AA \mapsto \frac{1}{\beta}A*A + \frac{\beta_1}{\beta}AbA $ & $11 \mapsto 1a1 $\\
$CC \mapsto \frac{\beta_1}{\beta \beta_2} CbC + \frac{\beta_1}{\sqrt{2} \beta_2} CdC $ & $22 \mapsto \frac{1}{\beta_1} 2a2 + \frac{\beta_2}{\beta_1}2c2 $\\
$EE \mapsto \frac{\beta_1}{\sqrt{2} \beta_{-1}}EdE+\frac{\beta_1}{\sqrt{2} \beta_{-1}} E\tilde{d}E + \frac{\sqrt{2}}{ \beta_{-1}} EfE  $ & $33 \mapsto \frac{\beta_1}{\sqrt{2} \beta \beta_2 \beta_2'}(\beta_1' + \frac{1}{\gamma'})3c3+ \frac{\sqrt{2}\beta\beta_1'}{\beta_2 \gamma'}3e3 $ \\
$\tilde{C} \tilde{C} \mapsto \frac{\sqrt{2} \gamma}{\beta_1 \beta_2'}\tilde{C} \tilde{d} \tilde{C} + \frac{\sqrt{2}}{\beta_2'}\tilde{C} \tilde{h} \tilde{C}$ & $44 \mapsto 4e4 $ \\
$GG \mapsto \frac{\sqrt{2}}{\beta_1} GfG + \frac{1}{\beta} G \tilde{b}G + \frac{1}{\beta} GhG $ & $55 \mapsto \frac{\beta_{-1}'}{\sqrt{2} \beta_2 \beta_1'} 5e5 + \frac{\beta_1}{\sqrt{2} \gamma}5\tilde{c}5+ \frac{\sqrt{2}}{\beta_2'}5g5 $ \\
$\tilde{A} \tilde{A} \mapsto \frac{\sqrt{2}}{\beta_2} \tilde{A}\tilde{h}\tilde{A} + \frac{1}{\beta}\tilde{A} \tilde{*} \tilde{A} $ & $66 \mapsto \frac{\sqrt{2}}{\beta_2}6g6+\frac{1}{\beta} 6 \tilde{a} 6 $ \\

\hline
\end{tabular*} 
\vskip2ex
We need to check that the these choices for $r_{\kappa} $ and $\bar{r}_{\kappa} $ are consistent.

\begin{lemma}
 The intertwiners $r_{\kappa} $ and $\bar{r}_{\kappa} $ defined as above, satisfy the cojugacy Equations \ref{con1} and \ref{con2}. 
\end{lemma}

\begin{proof}
 Since $dim(Id_N, \kappa \bar{\kappa})=dim(Id_M, \bar{\kappa}\kappa)=1 $ and $r_{\kappa}, \bar{r}_{\kappa} $ are isometries, Equations \ref{con1} and \ref{con2} are satisfied up to a unitary scalar. Since all the nonzero entries of the intertwiners are positive, that scalar must be $1$.
\end{proof}

The connection for $\rho $ was chosen in \cite{AH} so that all the gauge unitaries
between $\rho $ and $\kappa \bar{\kappa}$ corresponding to simple edges between
distinct vertices are the same. Let $v:\rho \rightarrow \kappa \bar{\kappa}$ be the isomtery determined by fixing those gauge unitaries to be $1$. We will need some coefficients of $v$. For an intertwiner $u$, and edges $xy \in \rho, xZy \kappa \bar{\kappa}$ we use the notation $u(xy,xZy) $ for the coefficient of $xZy$ in the image of $xy$ under $u$. (We will also use similar notation for coefficients of intertwiners of other connections.)

\begin{lemma}
 We have $v(\tilde{h}\tilde{h},\tilde{h}\tilde{A}\tilde{h} )=\frac{\beta_2}{\beta_1}=-v(bb, bAb) $.
\end{lemma}
\begin{proof}
 This is a straightforward computation using the methods of \cite{AH} and the fact that that all the simple gauge unitaries are $1$.
\end{proof}

Finally, let $w: \alpha \rho \alpha \kappa \rightarrow \rho \alpha \kappa$ be the isometry constructed in Theorem 3 of \cite{AH} . We recall the following coefficients.
\begin{lemma}
We have $w(b\tilde{b}h\tilde{h}\tilde{A},b*\tilde{*}\tilde{A})=w(\tilde{h}hffG,\tilde{h}\tilde{h}hG)=w(*\tilde{*}\tilde{h}hG,*b\tilde{b}G)= w(\tilde{*}*b\tilde{b}G,\tilde{*}\tilde{h}hG)=1=-w(b\tilde{b}ffG,bb\tilde{b}G) $.
\end{lemma}

Next, we define diagrams with the appropriate normalizations.

Let $$\hpic{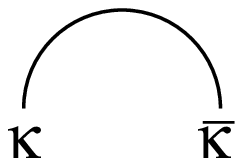} {0.4in} = \sqrt{\beta }r_{\kappa}, \hpic{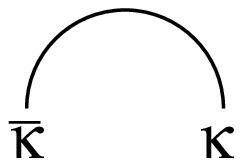} {0.4in} = \sqrt{\beta} \bar{r}_{\kappa} $$

$$\hpic{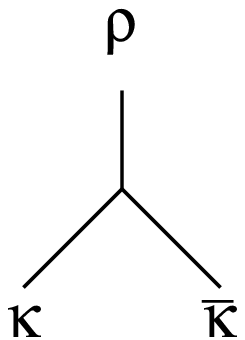} {0.8in} = \sqrt{\frac{\beta}{\beta_1}} v, \hpic{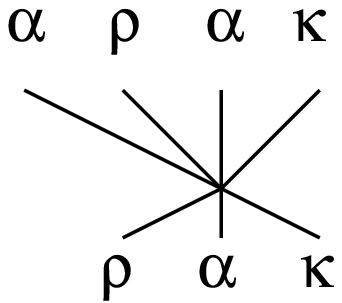} {0.8in} = w $$

For each of these diagrams define the diagram obtained by rotating by $\pi$ to be the adjoint. Define $$\hpic{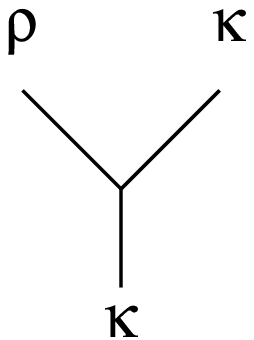} {0.8in} = \hpic{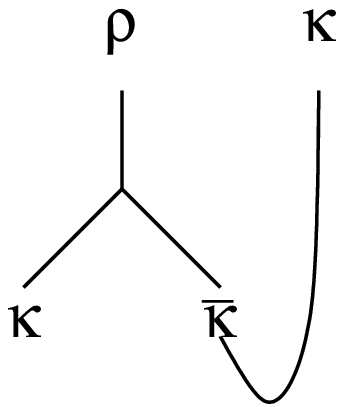} {0.8in} , \hpic{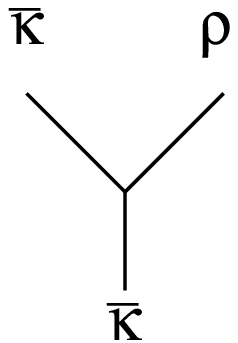} {0.8in} = \hpic{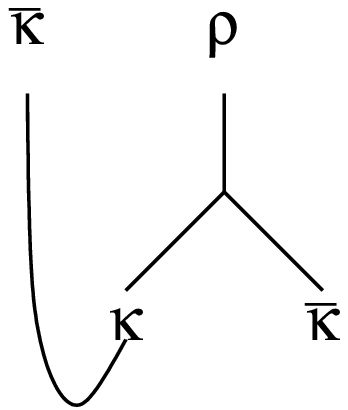} {0.8in} $$
and again define the diagrams obtained by rotating by $\pi$ to be the adjoint. 

\begin{lemma}
$ \hpic{rhok_k} {0.8in} = \hpic{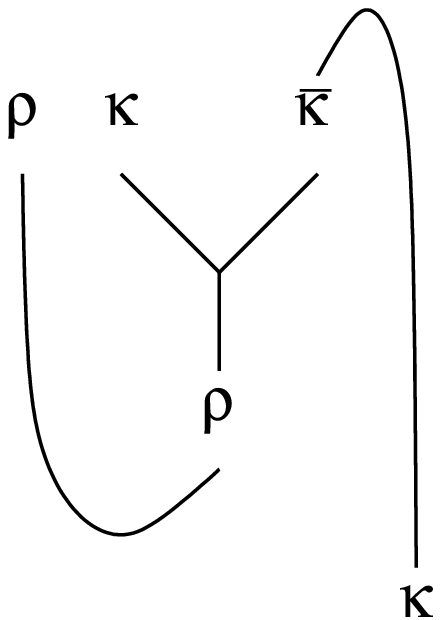} {1.0in} , \hpic{kbar_rho} {1.0in} = \hpic{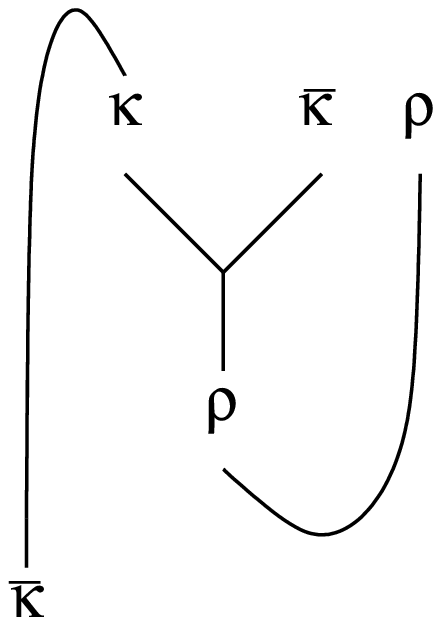} {0.8in} $
\end{lemma}
\begin{proof}
 By Frobenius reciprocity, the two diagrams in the first equation have the same norm (see \cite{GI}, Lemma 8.1 ), and they belong to the same one-dimensional space. Consider the edge  $*bA$ in $\rho \kappa$; both diagrams send it to $*A$ with positive coefficient, so they must represent the same intertwiner. The proof of the second equality is similar.
\end{proof}

This property allows us to unambiguously ``rotate'' these trivalent vertices.

Then let $$\hpic{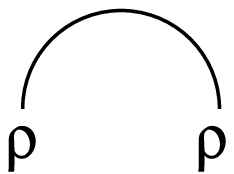} {0.4in} = \frac{\beta_1}{\beta} \hpic{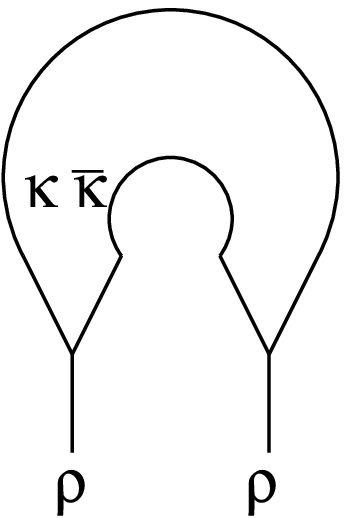} {1.2in} , \hpic{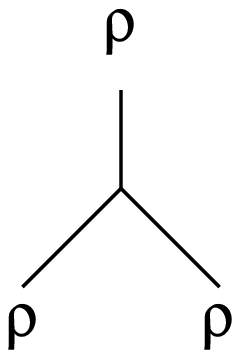} {0.8in} = \frac{\beta_1}{\beta_2} \hpic{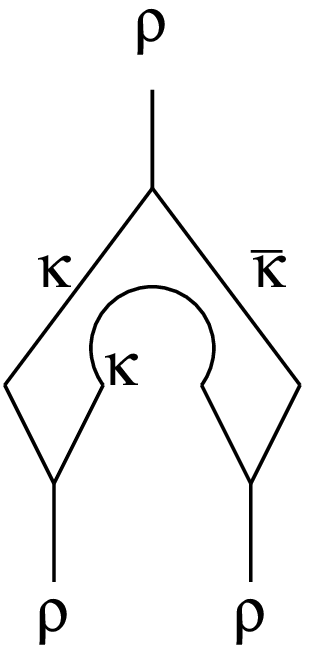} {1.4in} $$

$$\hpic{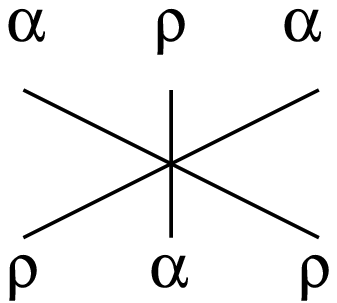} {0.8in} = \frac{\beta_1}{\beta} \hpic{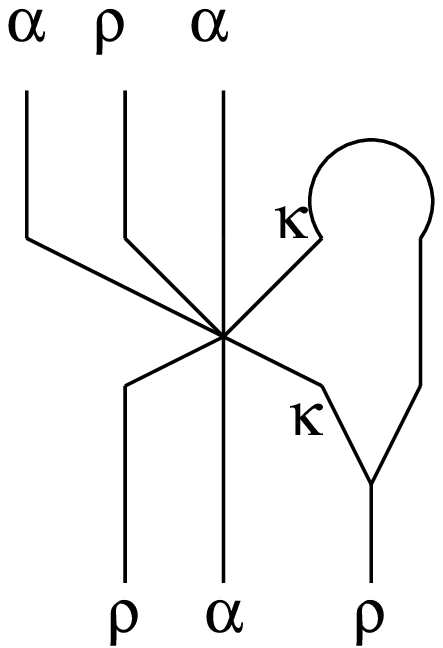} {1.2in} $$

Again, let each of the diagrams rotated by $\pi$ be the adjoint. Again we have $$\hpic{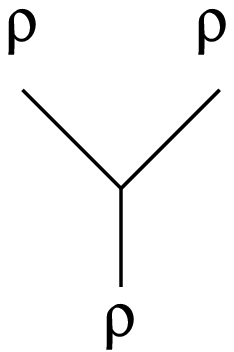} {0.8in}  = \hpic{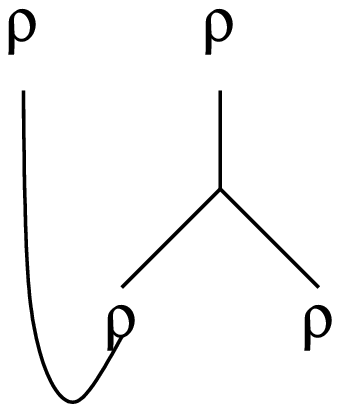} {0.8in} = \hpic{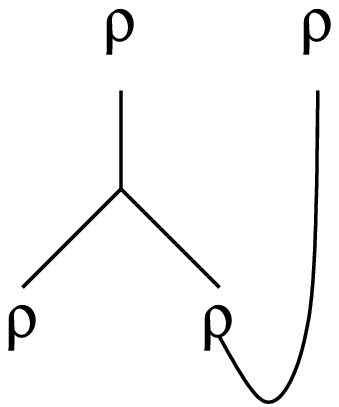} {0.8in} $$

We will need to compute coefficients of more complicated intertwiner diagrams, so we introduce the following formalism. By a vertex of an intertwiner diagram we will mean a crossing or relative extremum of the y-coordinate. The intertwiners represented by the vertices are called the elementary intertwiners of the diagram. Then each intertwiner diagram is a composition of elementary intertwiners tensored with identity morphisms. Such a diagram represents an intertwiner in $Hom(\lambda_1 ... \lambda_n, \mu_1 ... \mu_m ) $, where $\lambda_1, ... ,\lambda_n $ are the bimodules labeling the strings at the top of the diagram and $\mu_1, ... ,\mu_m$ label the strings at the bottom. Since an intertwiner is a map from the edge space of $\lambda_1 ... \lambda_n $ to that of $\mu_1 ... \mu_m $, the diagram can be evaluated on a specific edge in $\lambda_1 ... \lambda_n $ by ``following'' the edge vertically from top to bottom and composing the actions of the elementary intertwiners. 

\begin{definition}
 A state on an intertwiner diagram is a labeling of the strings of the diagram by edges in the corresponding bimodules and regions of the diagram by vertices such that each string labels an edge connecting the two adjacdent regions. (We imagine that the diagram is bounded by a box so one can't ``go around'' the top of the strings). A state determines a unique edge at each horizontal cross section of the diagram which doesn't contain  a vertex. The spin factor associated to a vertex is the coefficient of the corresponding elementary intertwiner from the edge directly above it to the edge directly below it. The value of a state is the product of the spin factors of all its vertices. 
\end{definition}

The following lemma is just an exercise in unraveling the definitions of states, intertwiners diagrams, and connections.

\begin{lemma}
 If $x$ is an edge in $\lambda_1 ... \lambda_n  $ and $y$ is an edge in $\mu_1 ... \mu_m $, then the $(x,y)$ coefficient of the intertwiner is the sum of the values of all states whose top horizontal cross section is $x$ and whose bottom horizontal cross section is $y$.
\end{lemma}

In the state diagrams that follow, the connection associated to each string will have a uniqe edge between the vertices of the adjacent regions, so we omit the labeling of the strings by edges. To avoid clutter, we also sometimes omit the labeling of the strings by the connections if it is clear from context.
 
We now compute a bunch of coefficients of various intertwiners that we will need later.

\begin{lemma}
We have the following coefficients: 
\begin{itemize}
\item[(a)] $\hpic{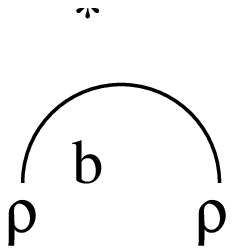} {0.5in} = \hpic{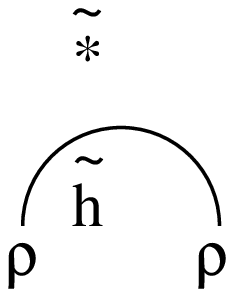} {0.6in} = \beta_1 , \hpic{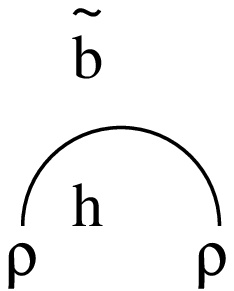} {0.6in} = \hpic{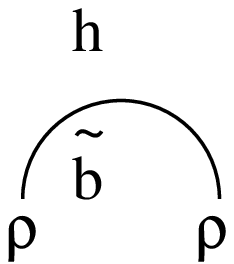} {0.5in} = 1,   $ \\
\item[(b)]  $- \hpic{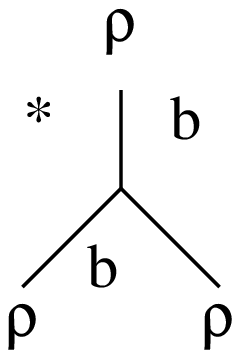} {0.7in } =\hpic{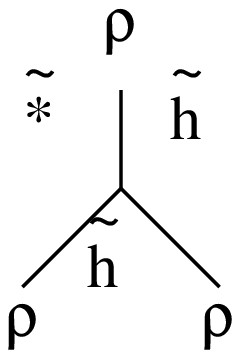} {0.7in} = \beta_2\sqrt{\frac{\beta_1 }{2}},  \hpic{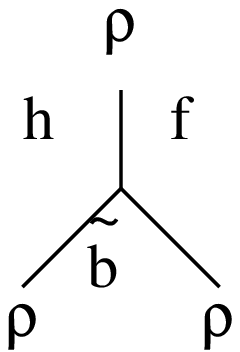} {0.7in} = \hpic{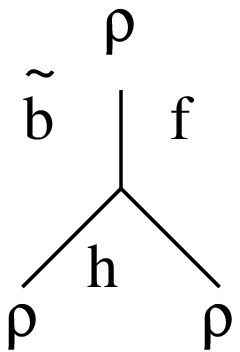} {0.7in} = \sqrt{\frac{\beta_1 }{2}}$, \\
\item[(c)] $\hpic{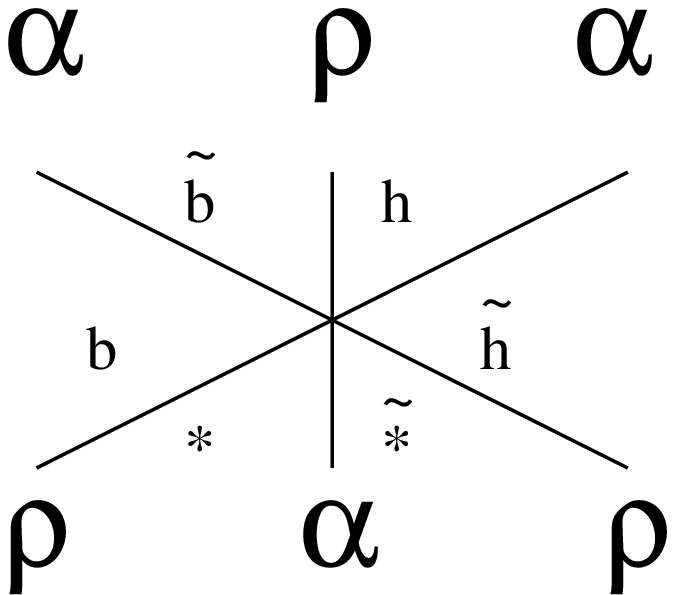} {0.6in} =  \hpic{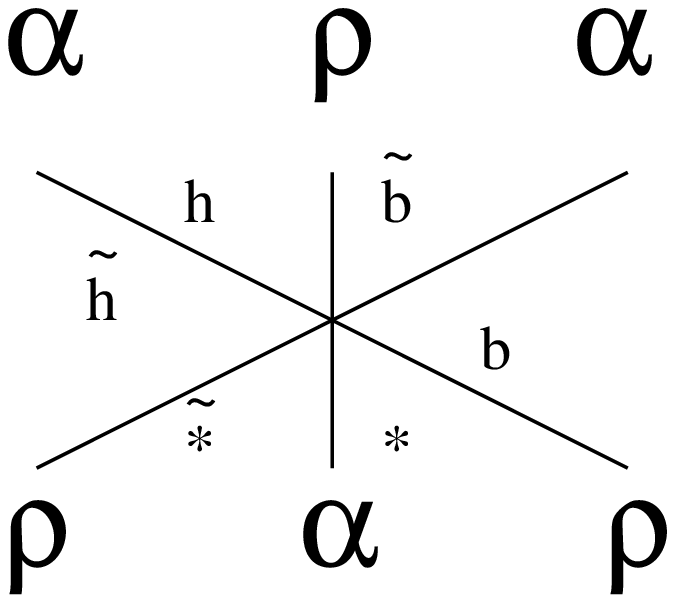} {0.6in} = \frac{1}{\sqrt{\beta_1}}, \hpic{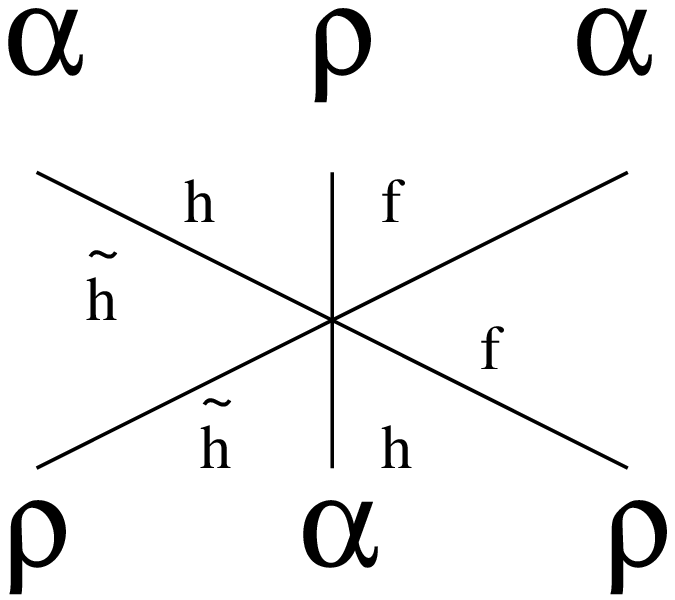} {0.6in } = - \hpic{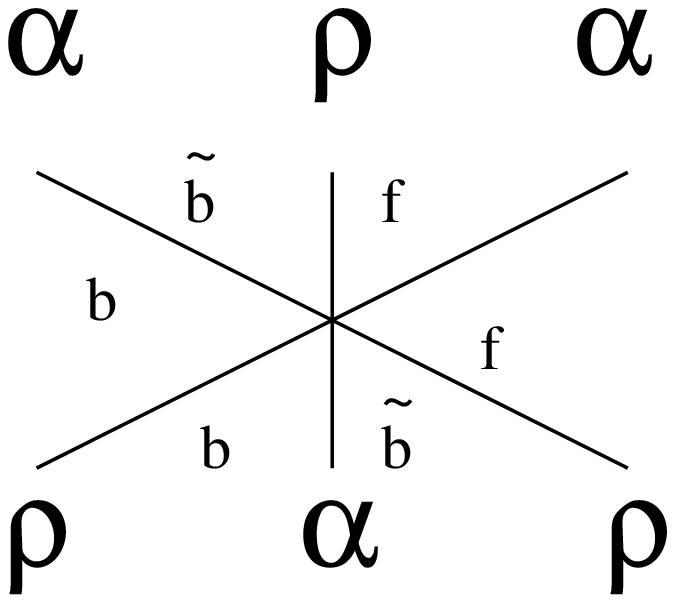} {0.6in} = \frac{\sqrt{\beta_1}}{\beta_2},  \hpic{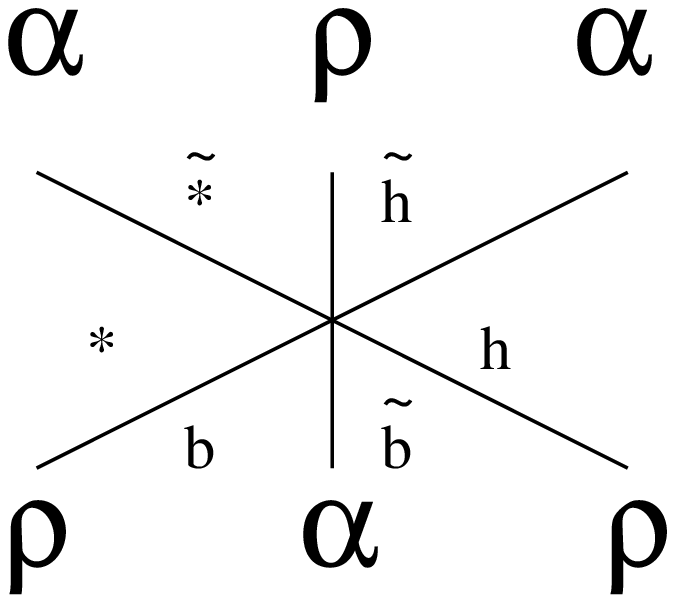} {0.6in} = \hpic{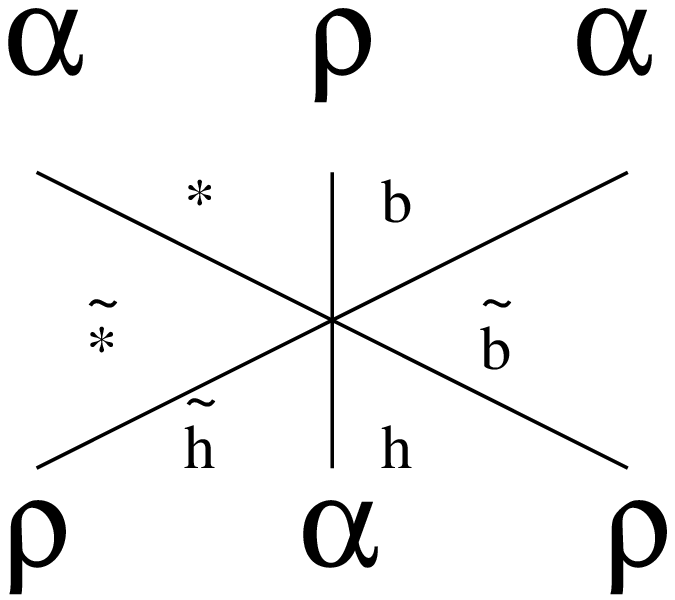} {0.6in} = \sqrt{\beta_1} $.
\end{itemize}
\end{lemma}

\begin{proof}
The idea behind all the computations is the same: we exress each intertwiner as a diagram whose elementary intertwiners are all known explicitly. Then the coefficients expand into states which can be evaluated. In principle we have to sum over all states compatible with the coefficient we are computing, but in practice each coefficient will determine a unique state. For each part of the lemma, we illustrate how the state breaks up into elementary intertwiners for the first coefficient computed, and then omit that step for the rest of the coefficients.

(a) $\hpic{rrho1} {0.5in} = \frac{\beta_1}{\beta} \hpic{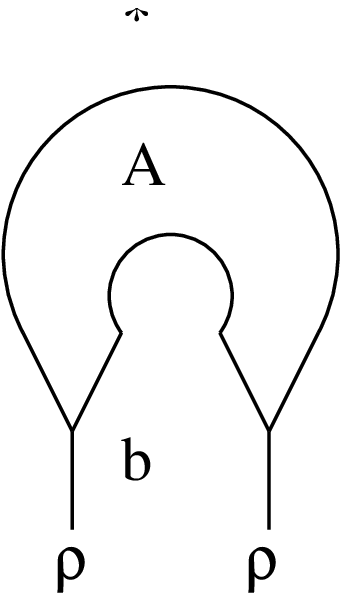} {0.8in} = \frac{\beta_1}{\beta} (\hpic{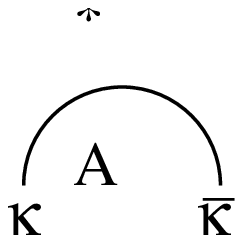} {0.4in} )( \hpic{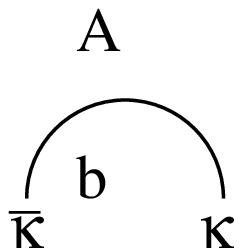} {0.4in} )( \hpic{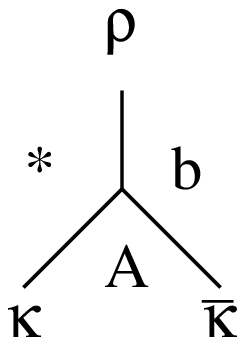} {0.5in} )( \hpic{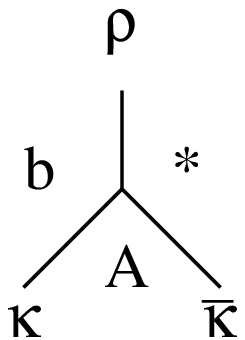} { 0.5in} ) \\
 \indent
= \frac{\beta_1}{\beta}\frac{\beta^2}{\beta_1} r_{\kappa}(**,*A*) \bar{r}_{\kappa}(AA,AbA)v(*b,*Ab)v(b*,bA*)= \beta(1)(\frac{\beta_1}{\beta})(1)(1) \\
 \indent =\beta_1$,

$\hpic{rrho2} {0.6in} = \frac{\beta_1}{\beta} \hpic{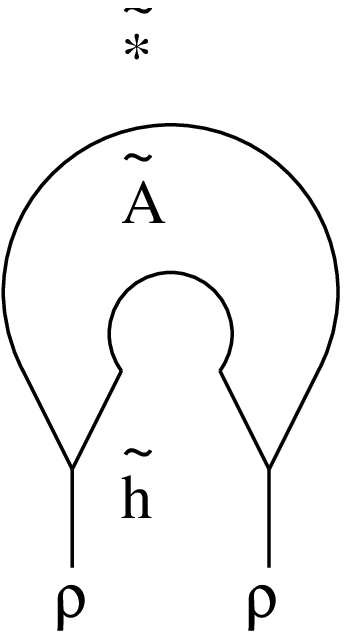} {0.8in} = \beta r_{\kappa}(\tilde{*}\tilde{*},\tilde{*}\tilde{A}\tilde{*}) \bar{r}_{\kappa}(\tilde{A}\tilde{A},\tilde{A}\tilde{h}\tilde{A})v(\tilde{*}\tilde{h},\tilde{*}\tilde{A}\tilde{h})v(\tilde{h}\tilde{*},\tilde{h}\tilde{A}\tilde{*}) \\
\indent = \beta(1)(\frac{\beta_1}{\beta})(1)(1)=\beta_1$,

$\hpic{rrho3} {0.6in} = \frac{\beta_1}{\beta} \hpic{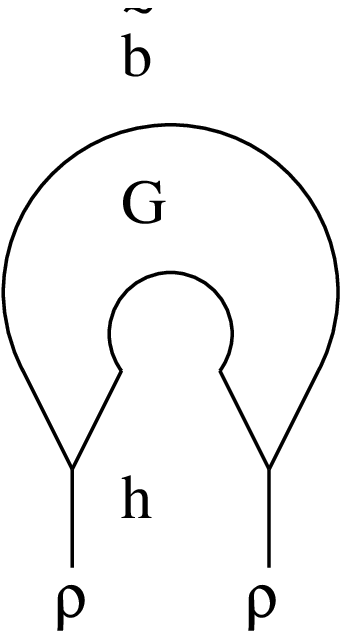} {0.8in} = \beta r_{\kappa}(\tilde{b}\tilde{b},\tilde{b}G\tilde{b}) \bar{r}_{\kappa}(GG,GhG)v(\tilde{b}h,\tilde{b}Gh)v(h\tilde{b},hG\tilde{b}) \\ \indent 
= \beta(\frac{1}{\beta})(1)(1)(1)=1$,

$\hpic{rrho4} {0.6in} = \frac{\beta_1}{\beta} \hpic{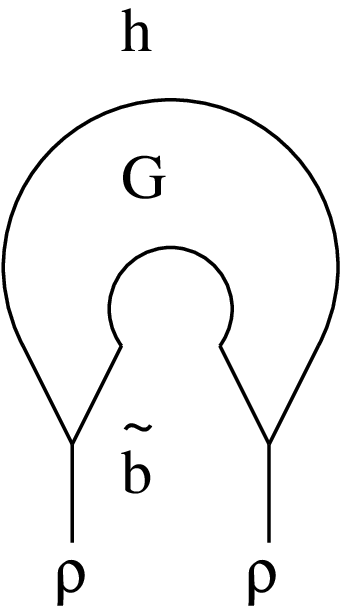} {0.8in} = \beta r_{\kappa}(hh,hGh) \bar{r}_{\kappa}(GG,G\tilde{b}G)v(h\tilde{b},hG\tilde{b})v(\tilde{b}h,\tilde{b}Gh) \\ \indent 
= \beta(\frac{1}{\beta})(1)(1)(1)=1$.
%
%
%
%
%

(b) $\hpic{rhorhorho1} {0.6in } = \frac{\beta_1}{\beta_2} \hpic{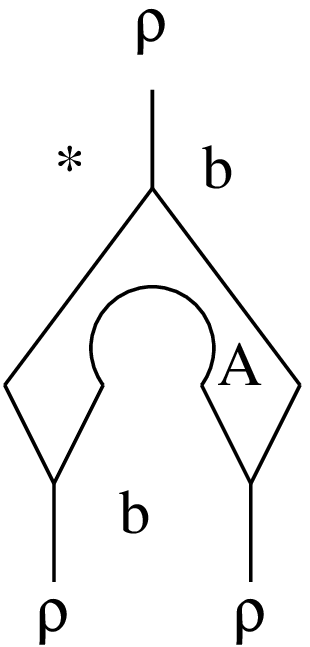} {0.8in} = \frac{\beta_1 }{\beta_2} (\hpic{rbarkappa1} {0.4in} )( \hpic{rhokk1} {0.5in} )( \hpic{rhokk1} {0.5in} )( \hpic{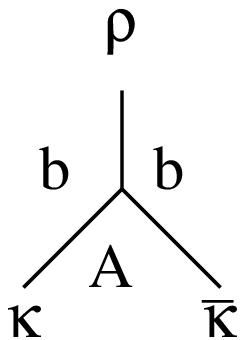} { 0.5in} ) \\ \indent
= \frac{\beta_1}{\beta_2} \frac{\beta^2}{\beta_1^{\frac{3}{2}}} \bar{r}_{\kappa}(AA,AbA )v(*b,*Ab)v(*b,*Ab)v(bb,bAb)=\beta\sqrt{\frac{\beta_1 }{2}}(\frac{\beta_1} {\beta})(1)(1)(-\frac{\beta_2 }{\beta_1 } ) \\ \indent 
=-\beta_2\sqrt{\frac{\beta_1 }{2}} $,

$\hpic{rhorhorho4} {0.6in } = \frac{\beta_1}{\beta_2} \hpic{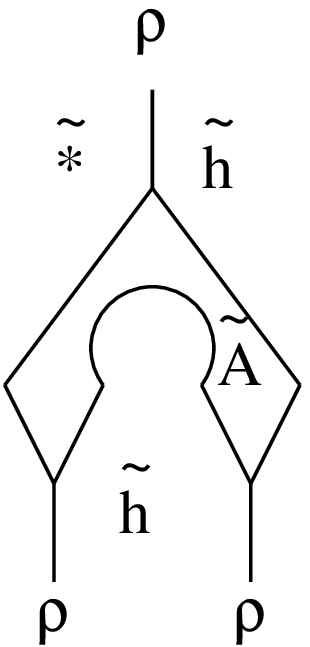} {0.8in} 
= \beta\sqrt{\frac{\beta_1 }{2}} \bar{r}_{\kappa}(\tilde{A}\tilde{A},\tilde{A}\tilde{h}\tilde{A} )v(\tilde{*}\tilde{h},\tilde{*}\tilde{A}\tilde{h})v(\tilde{*}\tilde{h},\tilde{*}\tilde{A}\tilde{h})v(\tilde{h}\tilde{h},\tilde{h}\tilde{A}\tilde{h}) \\ \indent
=\beta\sqrt{\frac{\beta_1 }{2}}(\frac{\beta_1} {\beta})(1)(1)(\frac{\beta_2 }{\beta_1 } )=\beta_2\sqrt{\frac{\beta_1 }{2}} $,

$\hpic{rhorhorho2} {0.6in } = \frac{\beta_1}{\beta_2} \hpic{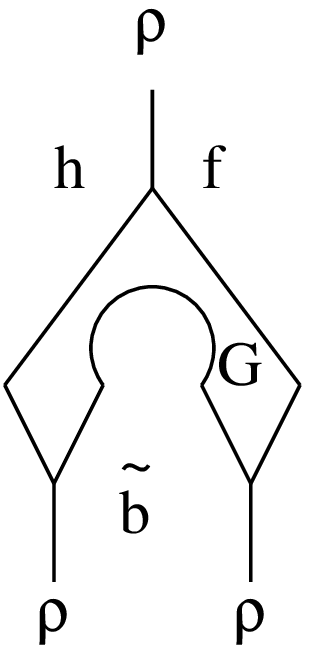} {0.8in} = \beta\sqrt{\frac{\beta_1 }{2}} \bar{r}_{\kappa}(GG, G\tilde{b}G)v(h\tilde{b},hG\tilde{b})v(\tilde{b}f,\tilde{b}Gf)v(hf,hGf)=\beta\sqrt{\frac{\beta_1 }{2}}(\frac{1} {\beta})(1)(1)(1)\\ \indent
=\sqrt{\frac{\beta_1 }{2}} $,

$\hpic{rhorhorho3} {0.6in } = \frac{\beta_1}{\beta_2} \hpic{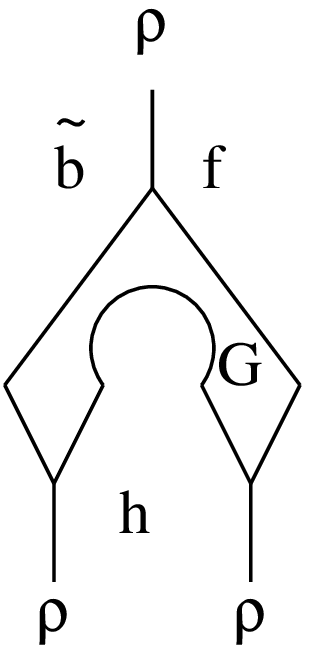} {0.8in} = \beta\sqrt{\frac{\beta_1 }{2}} \bar{r}_{\kappa}(GG, GhG )v(\tilde{b}h,\tilde{b}Gh)v(\tilde{b}f,\tilde{b}Gf)v(hf,hGf) \\ \indent
=\beta\sqrt{\frac{\beta_1 }{2}}(\frac{1} {\beta})(1)(1)(1)=\sqrt{\frac{\beta_1 }{2}} $.

%
%
%
%
%
%
%

(c) $\hpic{ararar1} {0.6in} = \frac{\beta_1}{\beta} \hpic{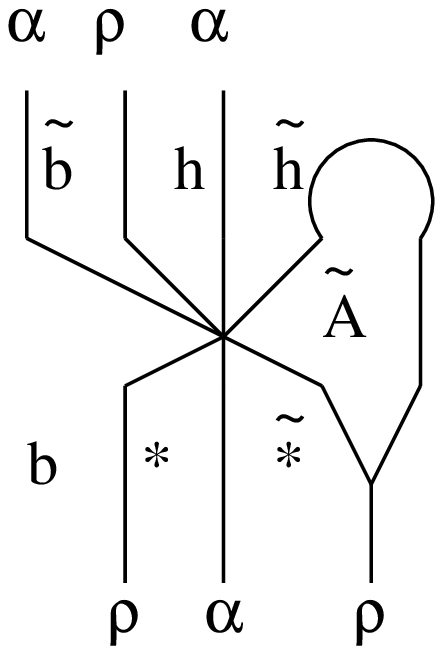} {0.8in} = \frac{\beta_1}{\beta} (\hpic{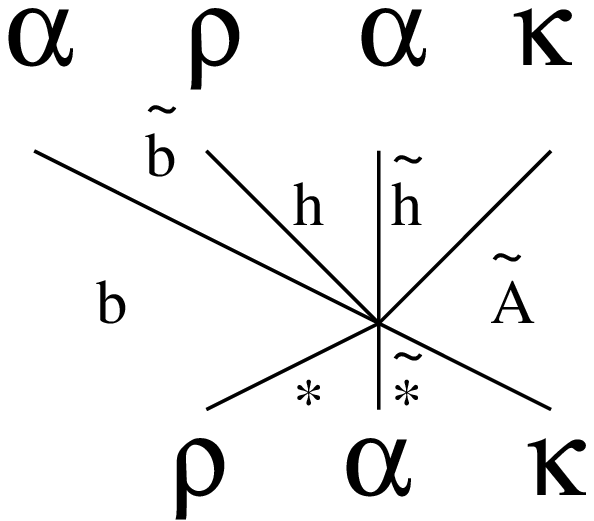} {0.6in} )( \hpic{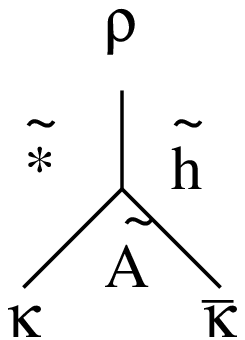} {0.5in} )( \hpic{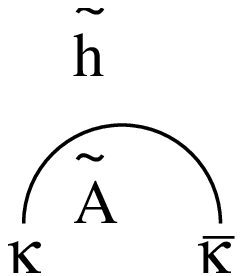} {0.5in} ) \\ \indent 
=\frac{\beta_1}{\beta} \frac{\beta}{\sqrt{\beta_1}} w(b\tilde{b}h\tilde{h}\tilde{A},b*\tilde{*}\tilde{A})v(\tilde{*}\tilde{h},\tilde{*}\tilde{A}\tilde{h})r(\tilde{h}\tilde{h},\tilde{h}\tilde{A}\tilde{h} ) 
= \sqrt{\beta_1} (1) (1 ) (\frac{1}{\beta_1} )=\frac{1}{\sqrt{\beta_1}}$,

$\hpic{ararar4} {0.6in} = \frac{\beta_1}{\beta} \hpic{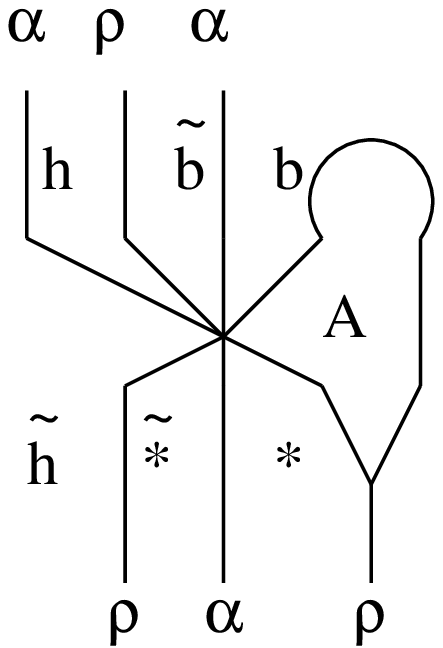} {0.8in} =\sqrt{\beta_1} w(\tilde{h}h\tilde{b}bA,\tilde{h}\tilde{*}*A)v(*b,*Ab)r(bb, bAb) \\ \indent
= \sqrt{\beta_1} (1) (1 ) (\frac{1}{\beta_1} )=\frac{1}{\sqrt{\beta_1}}$,

$\hpic{ararar5} {0.6in} = \frac{\beta_1}{\beta} \hpic{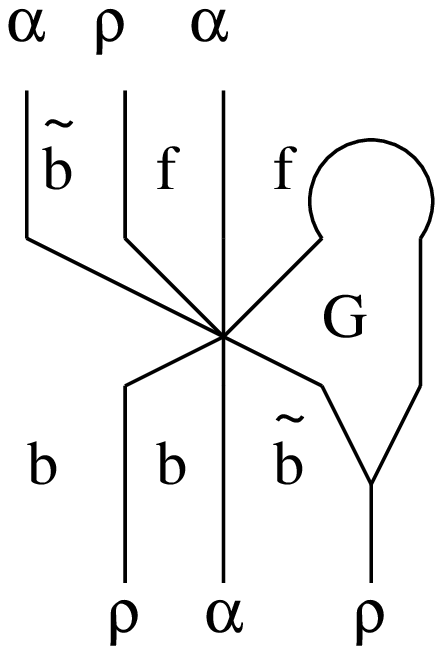} {0.8in} =\sqrt{\beta_1} w(b\tilde{b}ffG,bb\tilde{b}G)v(\tilde{b}f,\tilde{b}Gf)r(ff, fGf) \\ \indent 
= \sqrt{\beta_1} (-1) (1 ) (\frac{1}{\beta_2} )=-\frac{\sqrt{\beta_1}}{\beta_2}$,

$\hpic{ararar6} {0.6in} = \frac{\beta_1}{\beta} \hpic{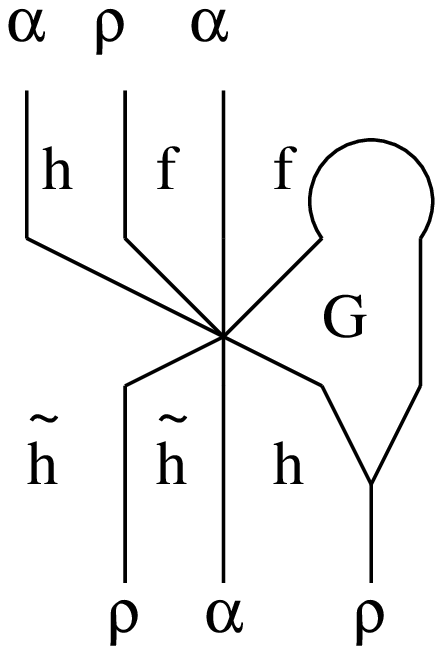} {0.8in} =\sqrt{\beta_1} w(b\tilde{b}ffG,bb\tilde{b}G)v(\tilde{b}f,\tilde{b}Gf)r(ff, fGf) \\ \indent 
= \sqrt{\beta_1} (1) (1 ) (\frac{1}{\beta_2} )=\frac{\sqrt{\beta_1}}{\beta_2}$,

$\hpic{ararar2} {0.6in} = \frac{\beta_1}{\beta} \hpic{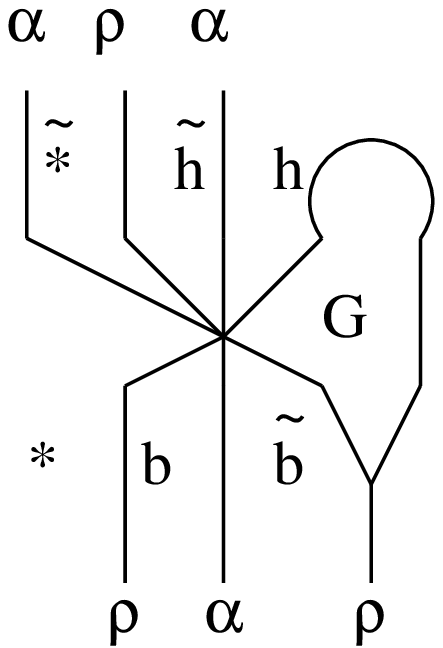} {0.8in} =\sqrt{\beta_1}w(*\tilde{*}\tilde{h}hG,*b\tilde{b}G)v(\tilde{b}h,\tilde{b}Gh)r(hh,hGh ) \\ \indent 
=\sqrt{\beta_1}( 1)( 1)(1 )=\sqrt{\beta_1}$,

$\hpic{ararar3} {0.6in} = \frac{\beta_1}{\beta} \hpic{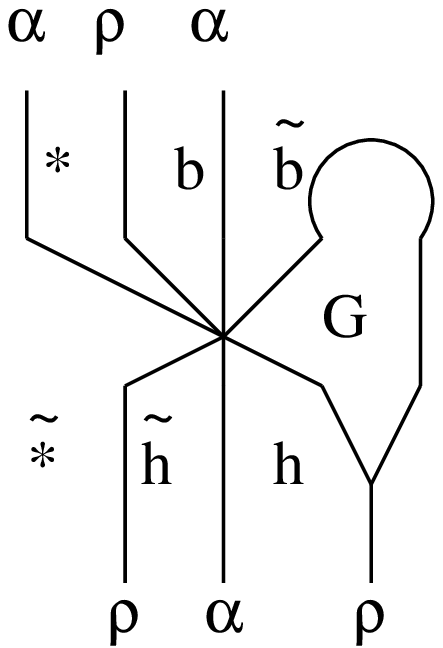} {0.8in}  \\ \indent 
 =\sqrt{\beta_1}w(\tilde{*}*b\tilde{b}G,\tilde{*}\tilde{h}hG)v(h\tilde{b},hG\tilde{b})r(\tilde{b}\tilde{b},\tilde{b}G\tilde{b} )= \sqrt{\beta_1}(1 ) (1 ) (1 )=\sqrt{\beta_1}$.

%
%
%

\end{proof}

%
%
%
%
%
%
%
%

\begin{lemma} \label{eq1}
 We have $\hpic{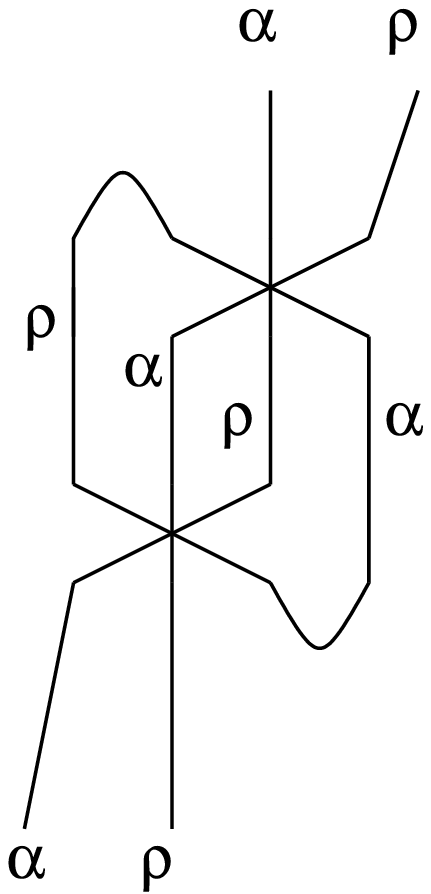} {1.2in} = \beta_1 Id_{\alpha \rho}$, $\hpic{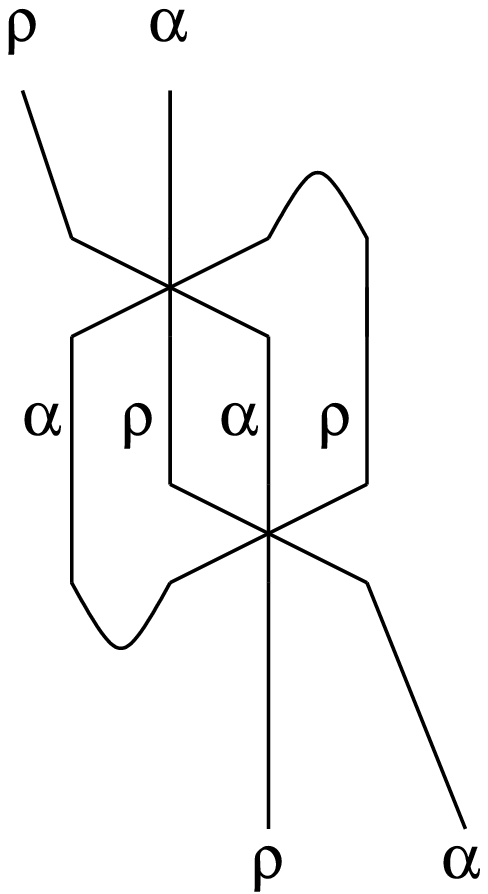} {1.2in} = \beta_1 Id_{\rho \alpha}$.
\end{lemma}
\begin{proof}
The left hand side of each equation is a scalar, so we can simply evaluate the unique state comptabile with any given edge.
For the first equation we have: \\
 $\hpic{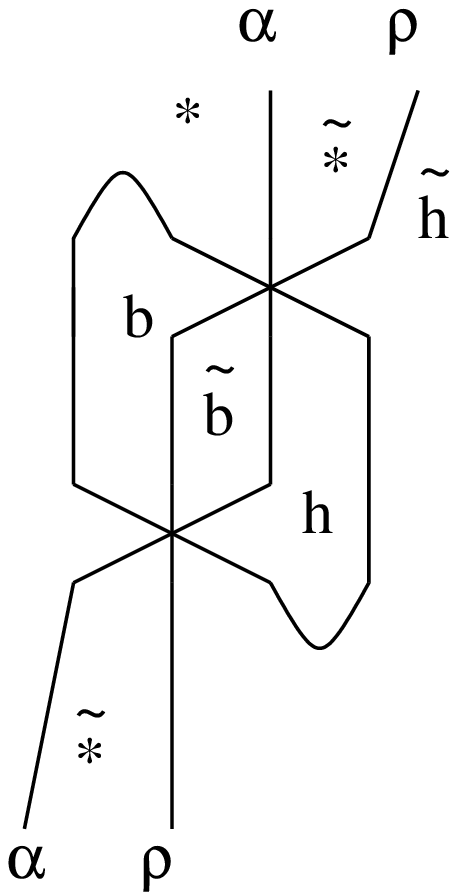} {1.2in} = ( \hpic{rrho1} {0.5in } ) ( \hpic{ararar1} {0.5in} ) ( \hpic{ararar2} {0.5in} ) = \beta_1 \frac{1}{\sqrt{\beta_1}} \sqrt{\beta_1}$. \\
And for the second: \\
$\hpic{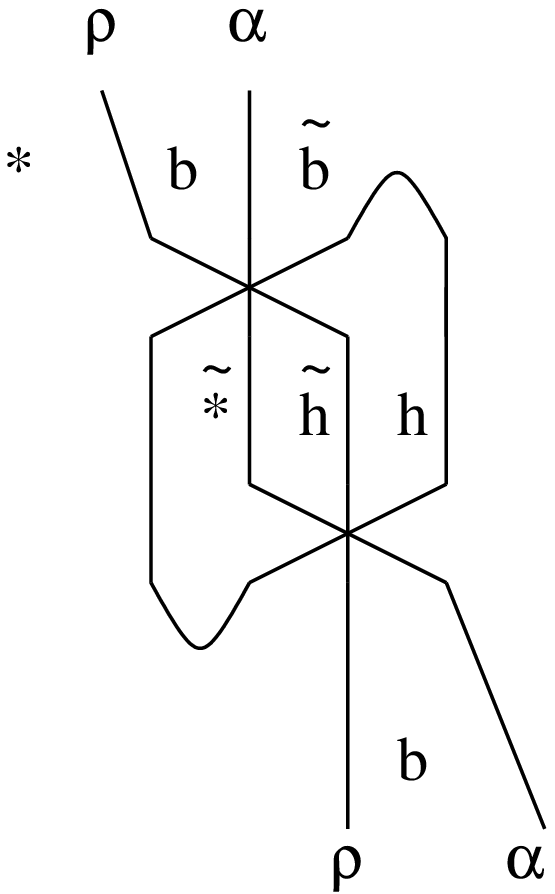} {1.2in} = ( \hpic{rrho3} {0.5in } ) ( \hpic{ararar2} {0.5in} ) ( \hpic{ararar3} {0.5in} ) = 1 \sqrt{\beta_1}\sqrt{\beta_1}$. 
\end{proof}

\begin{corollary} \label{eq4}
 We have \\ 
 $ \hpic{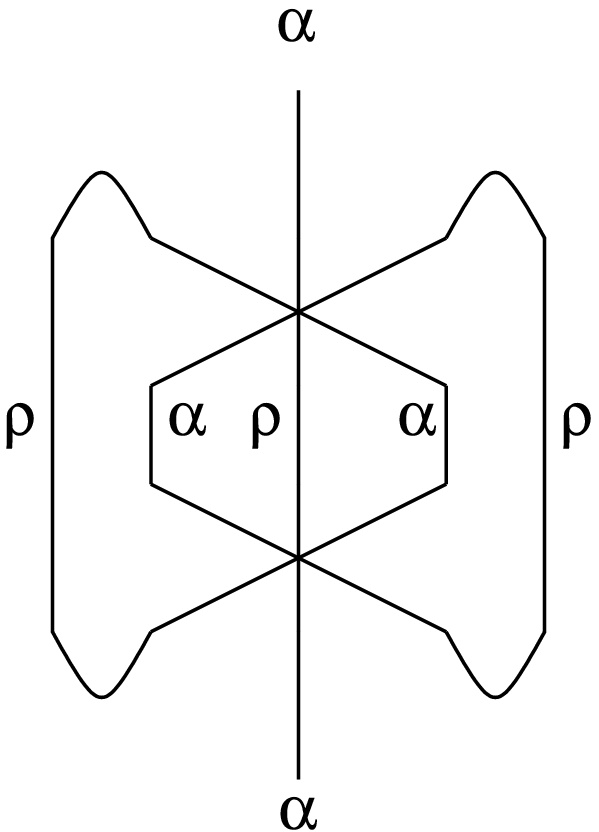} {1.2in} = \beta_1^3 Id_{\alpha}$.
\end{corollary}

\begin{lemma} \label{eq3}
 We have $\hpic{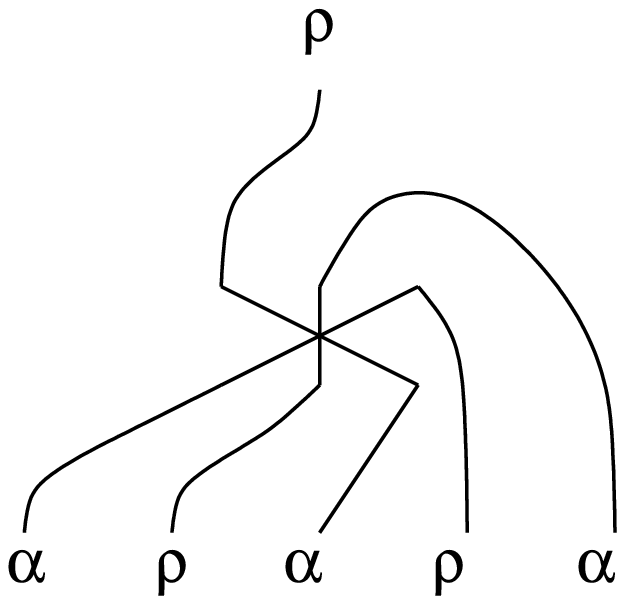} {1.0in} = \hpic{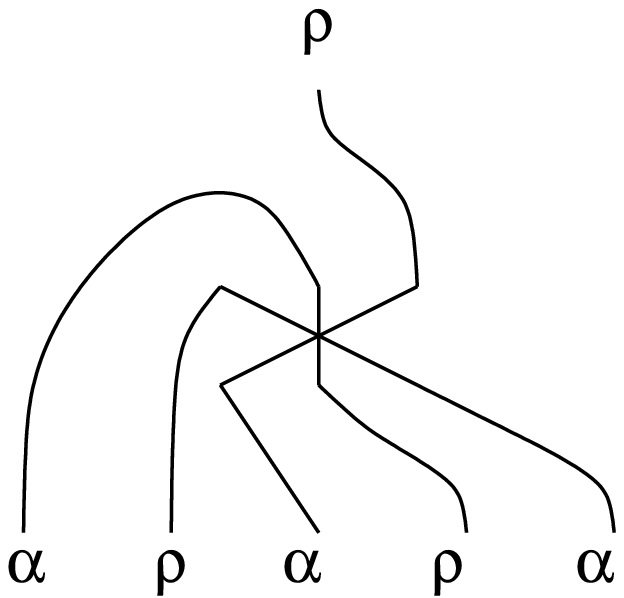} {1.0in} $.
\end{lemma}
\begin{proof}
Since $dim(\rho,\alpha\rho\alpha\rho\alpha )=1$, we can compare the two sides of the equations using any nonzero coefficient. We choose the coefficient corresponding to the edges $(*b, *\tilde{*}\tilde{h}h\tilde{b}b) $, which admits a unique compatible state for each of the diagrams in the equation. For the left hand side we have:

$\hpic{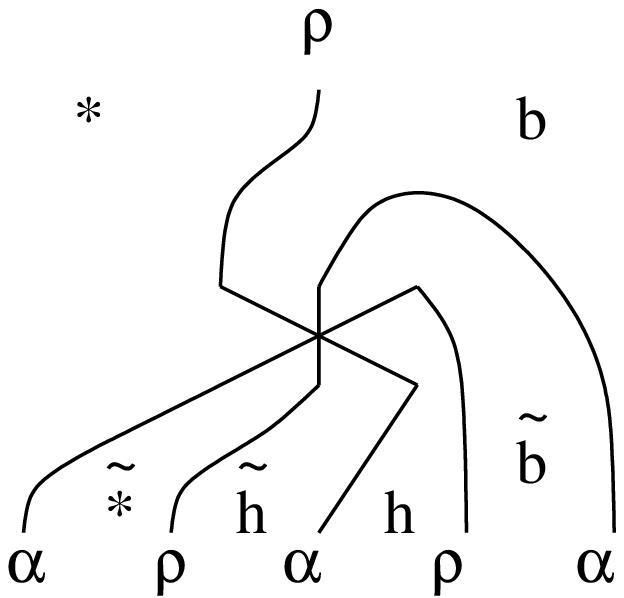} {1.0in} = (\hpic{rrho4} {0.5in} ) (\hpic{ararar2} {0.5in} ) = 1 \sqrt{\beta_1}$,

and for the right hand side:

$\hpic{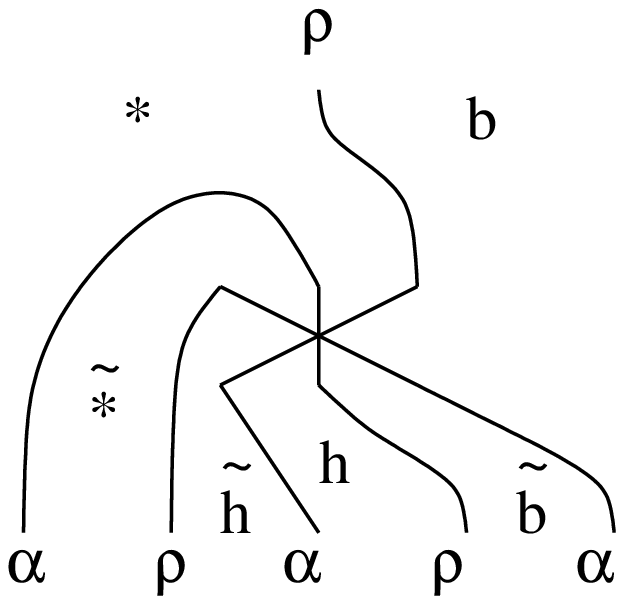} {1.0in} = (\hpic{rrho2} {0.5in} ) (\hpic{ararar4} {0.5in} ) = \beta_1 \frac{1}{\sqrt{\beta_1}}$.
\end{proof}

\begin{lemma} \label{eq2}
 $\hpic{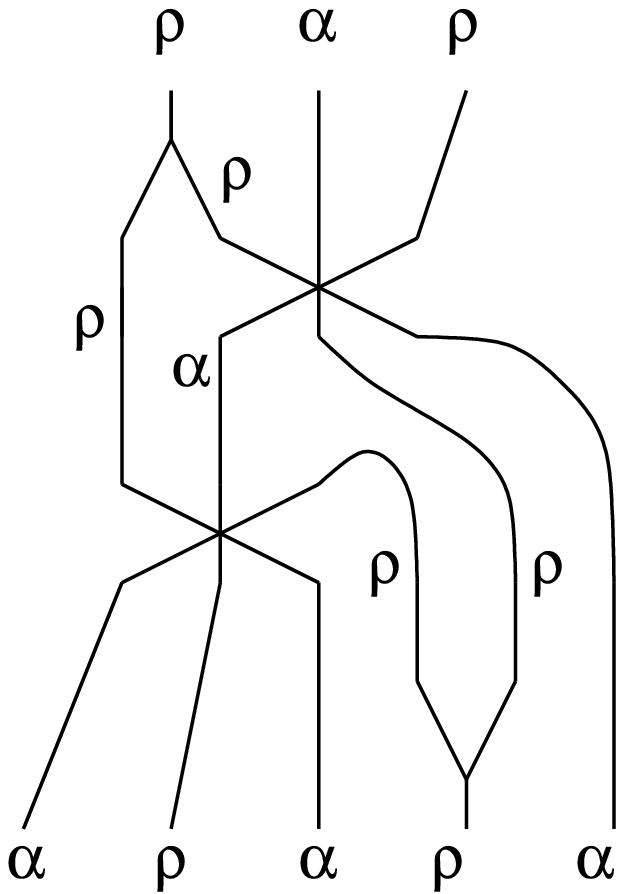} {1.5in} = \hpic{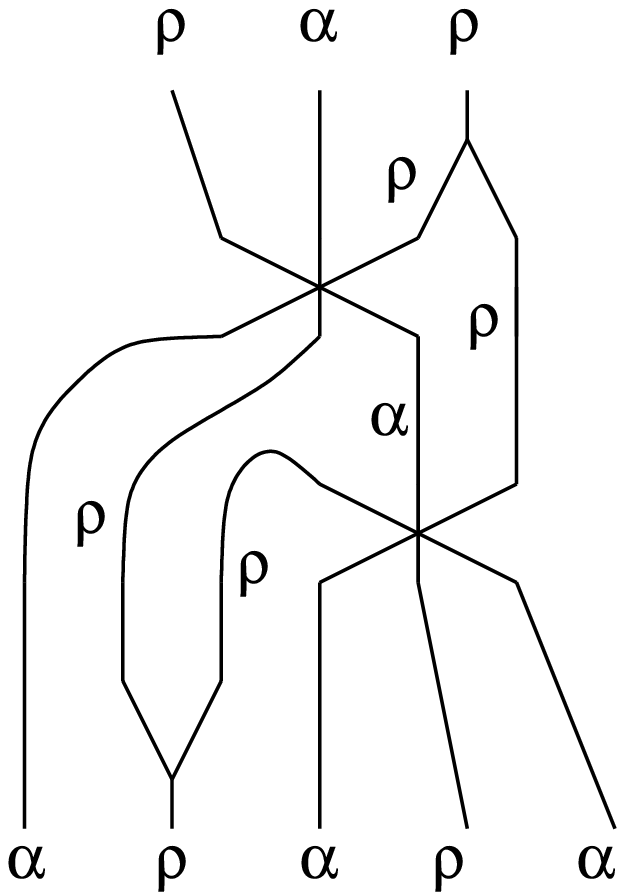} {1.5in} $
\end{lemma}

\begin{proof}
Again, $dim(\rho  \alpha \rho, \alpha \rho \alpha \rho \alpha ) = 1 $, so we can compare the two diagrams using any nonzero coefficient. We choose $(*b\tilde{b}f,*\tilde{*}\tilde{h}hff) $, and find, for the left hand side:

$\hpic{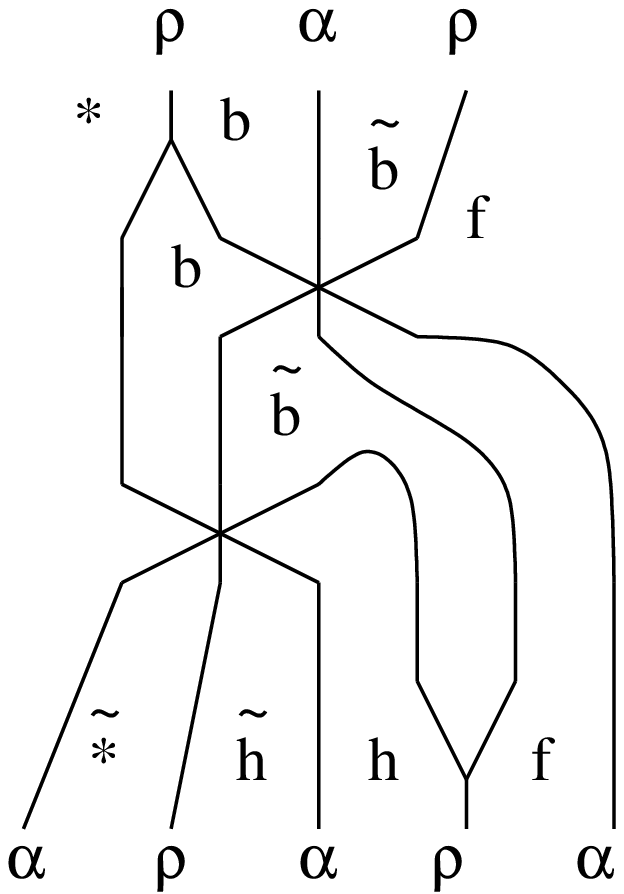} {1.5in} = (\hpic{ararar5} {0.5in} ) (\hpic{ararar2} {0.5in} )(\hpic{rhorhorho1} {0.5in} )(\hpic{rhorhorho2} {0.5in} )  \\ \indent  
=\\ -\frac{\sqrt{\beta_1}}{\beta_2} \sqrt{\beta_1} \beta_2 \sqrt{\frac{\beta_1}{ 2}} \sqrt{\frac{\beta_1}{ 2}}$, \\

and for the right hand side:

$\hpic{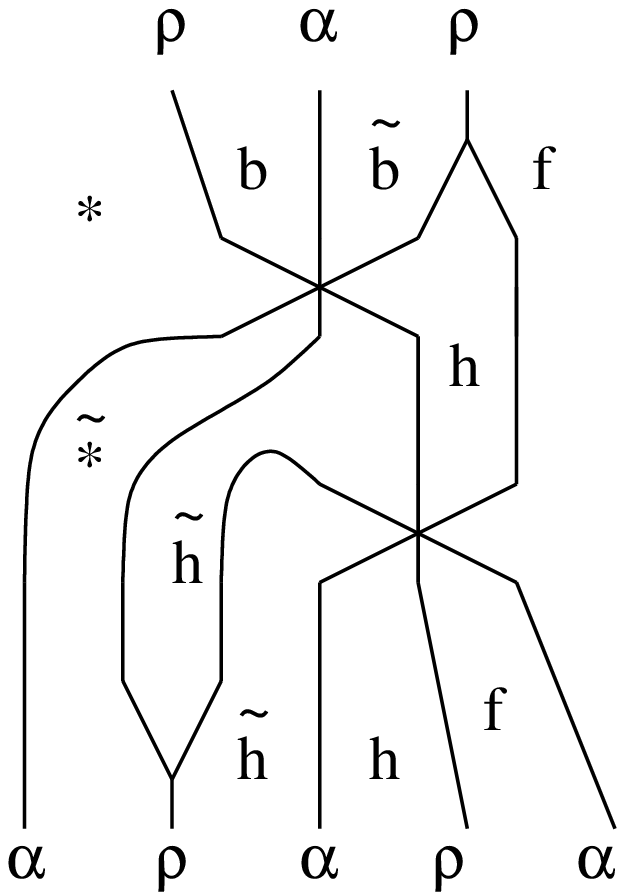} {1.5in} = (\hpic{ararar6} {0.5in} ) (\hpic{ararar2} {0.5in} )(\hpic{rhorhorho3} {0.5in} ) (\hpic{rhorhorho4} {0.5in} ) \\ \indent  
=  \frac{\sqrt{\beta_1}}{\beta_2} \sqrt{\beta_1} \sqrt{\frac{\beta_1}{ 2}} (-\beta_2 \sqrt{\frac{\beta_1}{ 2}})  $.

\end{proof}

\begin{lemma} \label{latelemma}
 We have $\hpic{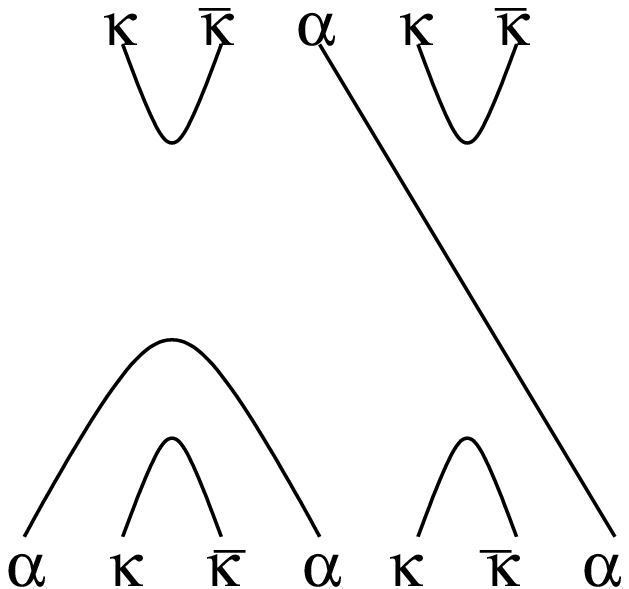} {1.2in} = \hpic{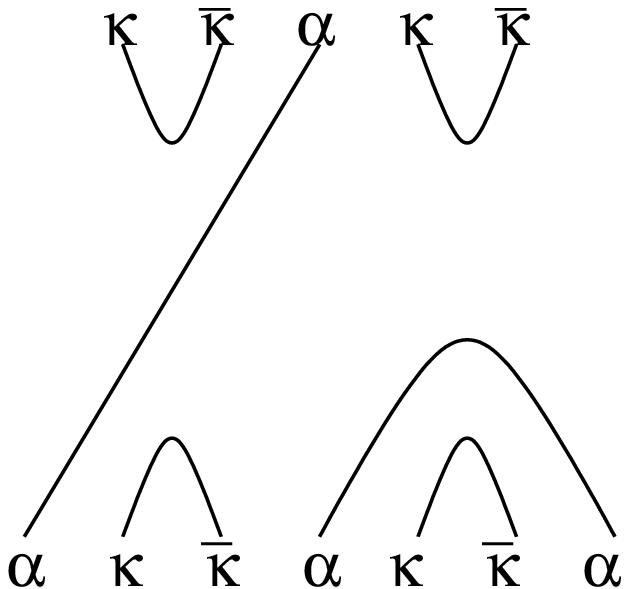} {1.2in} $.
\end{lemma}
\begin{proof}
 This time $dim(\kappa \bar{\kappa} \alpha \kappa \bar{\kappa},\alpha \kappa \bar{\kappa} \alpha \kappa \bar{\kappa} \alpha) = 4$, so evaluating a single nonzero coefficient on each side is insufficient. However, the only compatible states of these diagrams are of the form \hpic{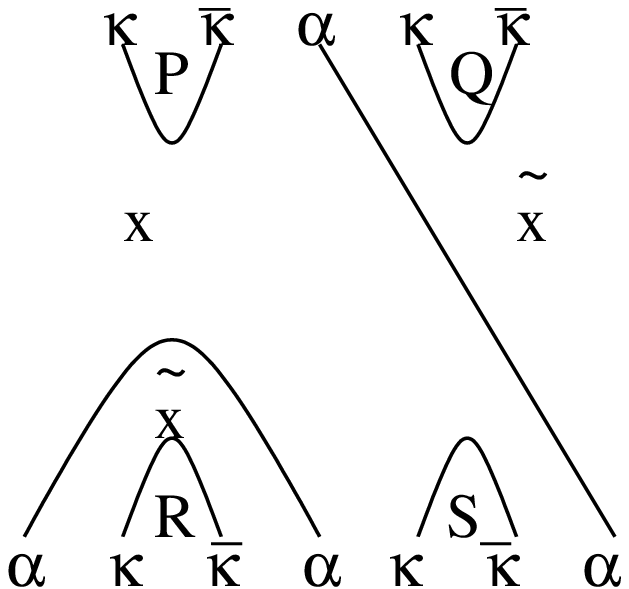} {1.2in} (for the left hand side) and \hpic{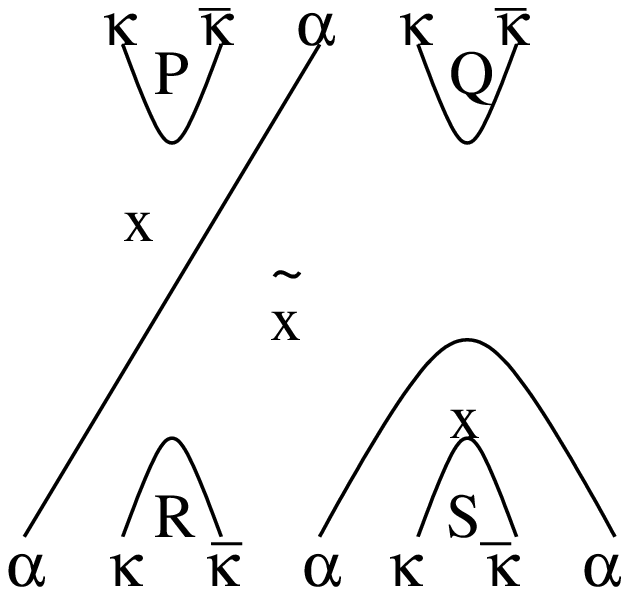} {1.2in} (for the right hand side), where $x,P,Q,R,S $ are some vertices from the appropriate graphs.

Since $\hpic{Rlt1symstate} {1.2in} = (\hpic{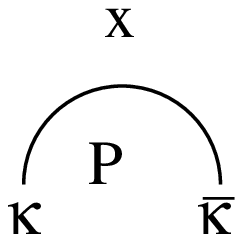} {0.5in} )(\hpic{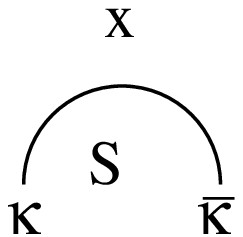} {0.5in}  )(\hpic{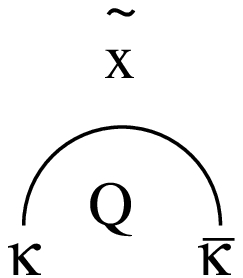} {0.5in}  ) (\hpic{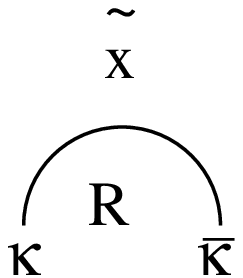} {0.5in}  ) = \hpic{Rrt1symstate} {1.2in} $, the two intertwiners are the same.
\end{proof}

\section{A quadrilateral}
The following lemma was proved for the Haagerup category in \cite{GI}. As the proof for the Asaeda-Haagerup category is identical we omit it here.
\begin{lemma} \label{skein}
 We have \hpic{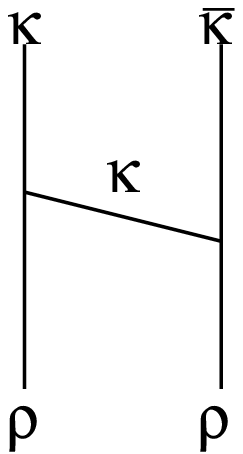} {0.8in} = \hpic{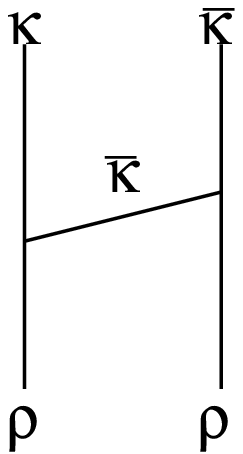} {0.8in} = $\frac{1}{\beta_1}$  \hpic{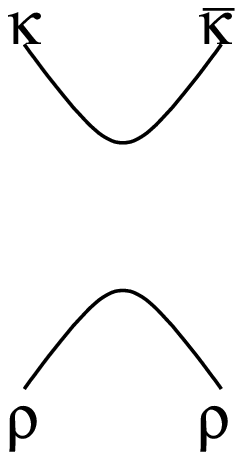} {0.8in} + $\frac{\beta_2}{\beta_1}$ \hpic{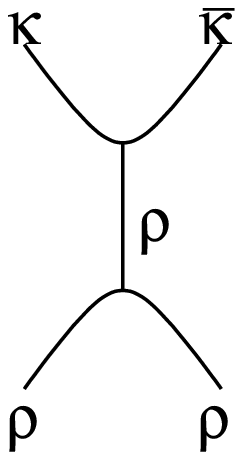} {0.8in}  .
\end{lemma}

We want to prove that $Id_N  \oplus \bar{\kappa} \alpha \kappa$ admits a Q-system; to do this we need to show that there are isometries $R \in ( Id_M,
\bar{\kappa}
\alpha \kappa \bar{\kappa} \alpha \kappa), S \in ( \bar{\kappa} \alpha \kappa,
\bar{\kappa} \alpha \kappa \bar{\kappa} \alpha \kappa )$ satisfying (1) and (2)
of \ref{qsystem_1plus} with $d=\beta^2 $.

Let $R=\frac{1}{\beta} \hpic{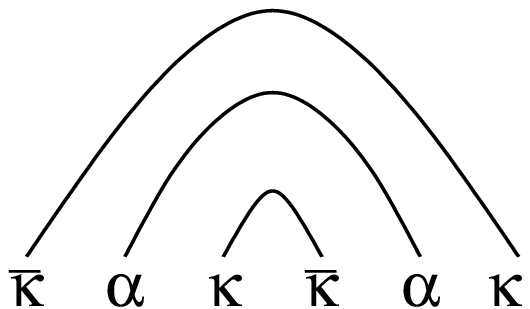} {0.7in} $ and $S =  \frac{1}{\sqrt{\beta} } \hpic{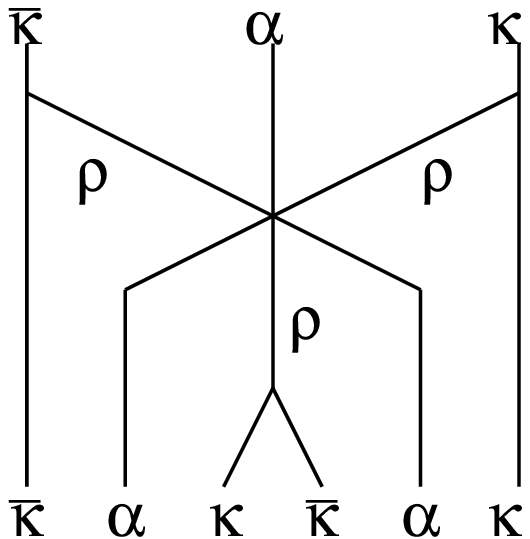} {1.0in} $.

\begin{lemma}\label{l1}
 $R$ and $S$ are isometries.
\end{lemma}

\begin{proof}
 
$R$ is clearly an isometry. We have \\
$S^*S= \frac{1}{\beta} \hpic{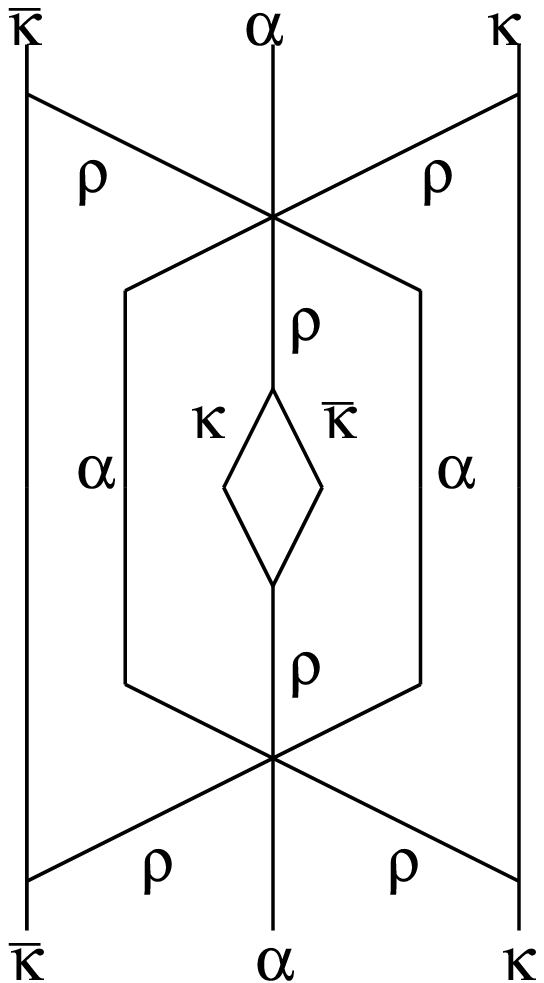} {1.2in} = \frac{1}{\beta_1} \hpic{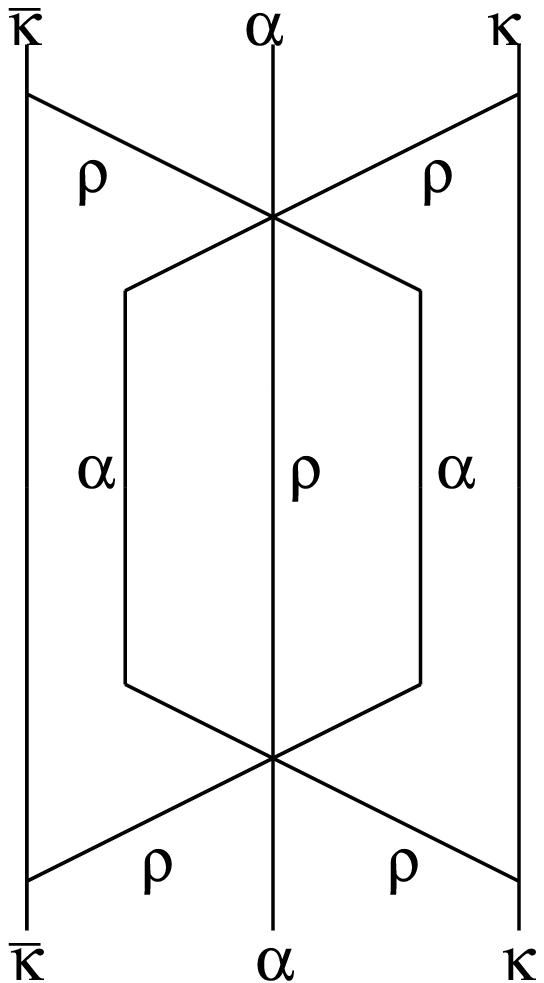} {1.2in}  =\frac{1}{\beta_1} (\frac{1}{\beta_1^2} \hpic{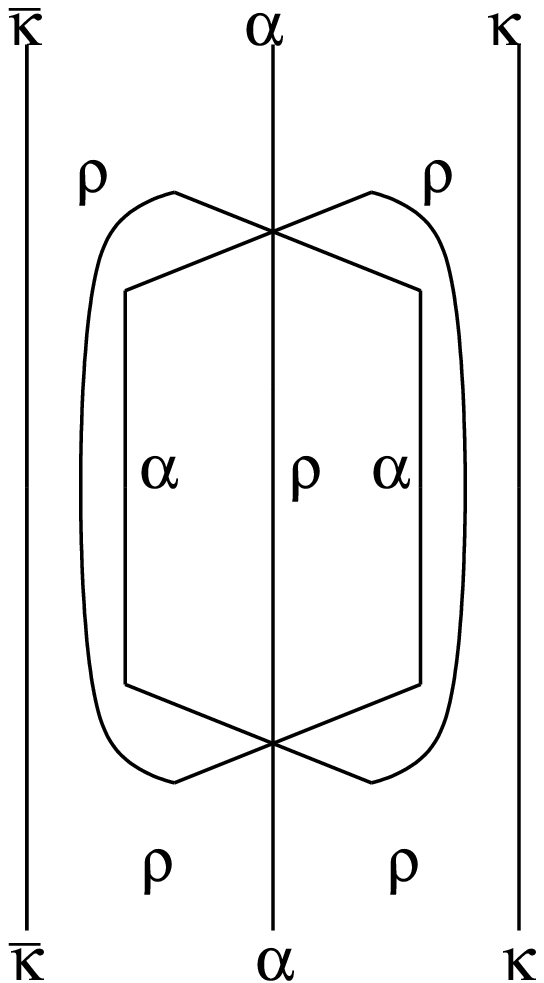} {1.2in} + \frac{\beta_2}{\beta_1^2} \hpic{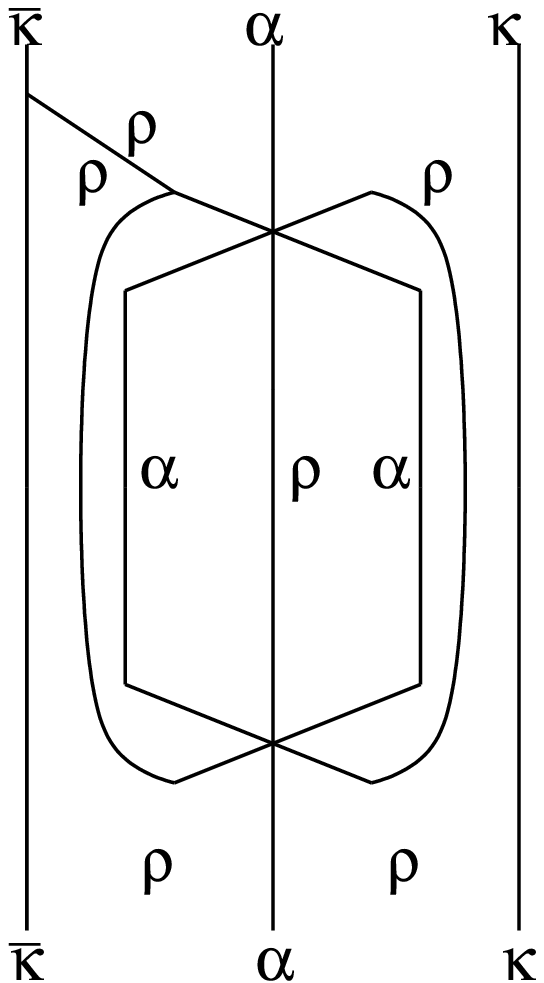} {1.2in}  \\
+ \frac{\beta_2^2}{\beta_1^2} \hpic{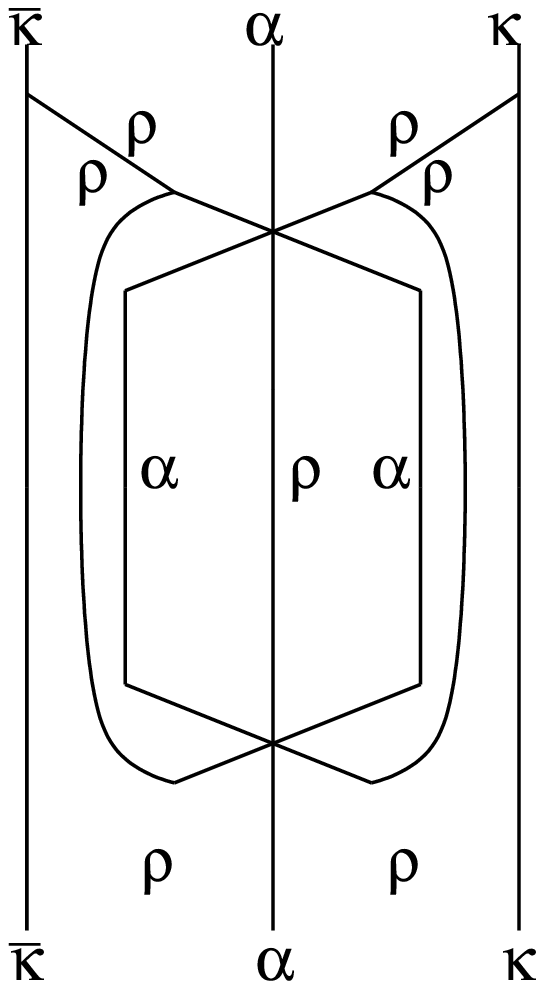} {1.2in} + \frac{\beta_2}{\beta_1^2} \hpic{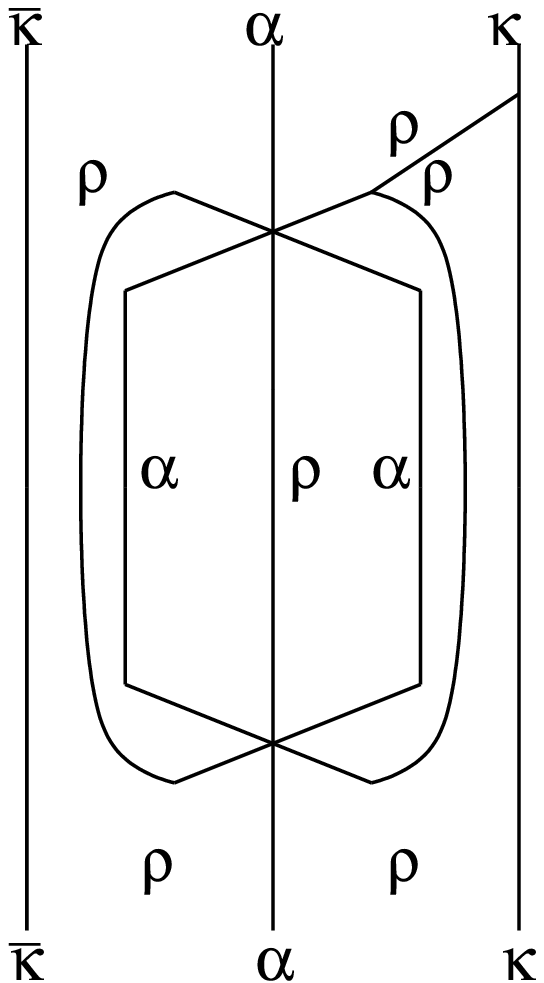} {1.2in} ) = \frac{1}{\beta_1^3} \beta_1^3 Id_{\bar{\kappa} \alpha \kappa} = Id_{\bar{\kappa} \alpha \kappa} $,  \\
where we have used \ref{skein} twice to split the diagram into four, and \ref{eq4} to evaluate the only nonzero term. 
\end{proof}

\begin{lemma}\label{l2}
We have $(S \otimes Id_{\sigma} ) \circ R = ( Id_{\sigma} \otimes S) \circ R $.
\end{lemma}

\begin{proof}
 We have $(S \otimes Id_{\sigma} ) \circ R = \frac{1}{\beta^{\frac{3}{2}}} \hpic{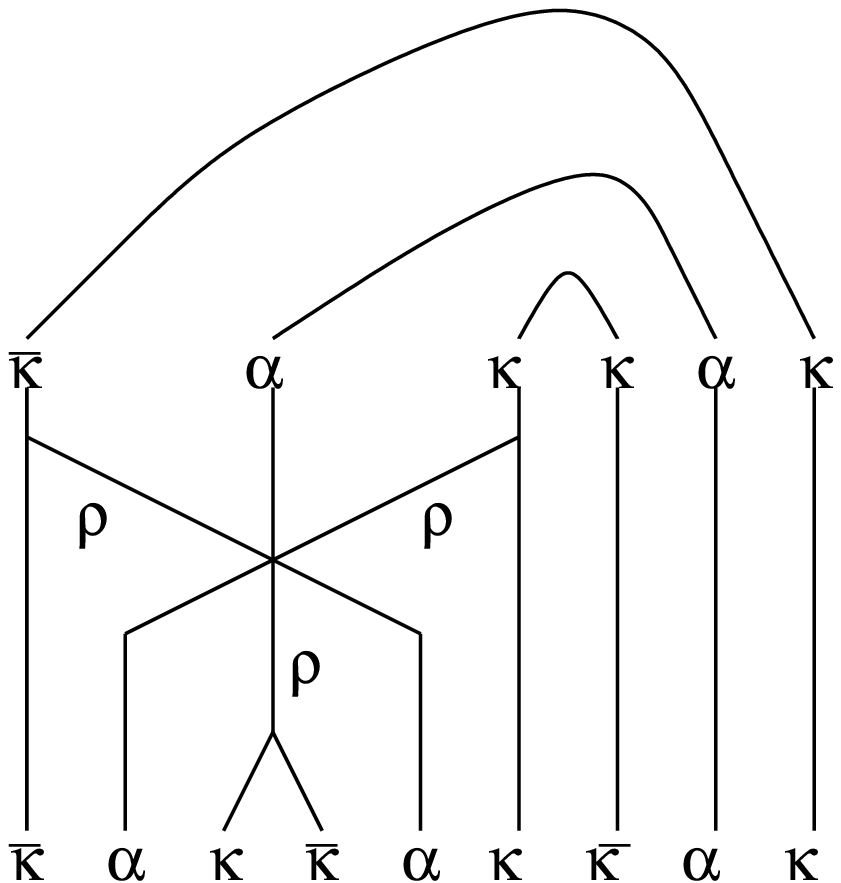} {1.2in} $. Its image under the linear isomorphism \hpic{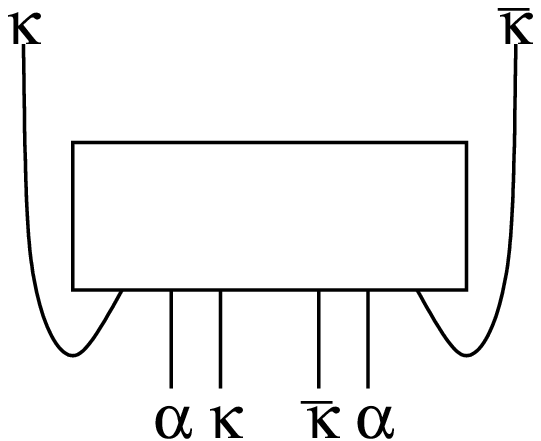} {0.8in} is $\frac{1}{\beta^{\frac{3}{2}}} \hpic{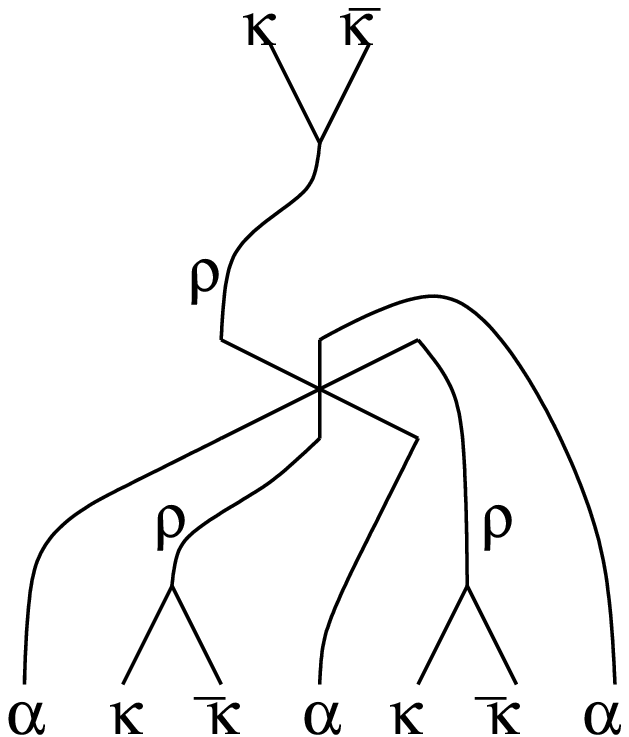} {1.2in} $. On the other hand, $( Id_{\sigma} \otimes S) \circ R = \frac{1}{\beta^{\frac{3}{2}}} \hpic{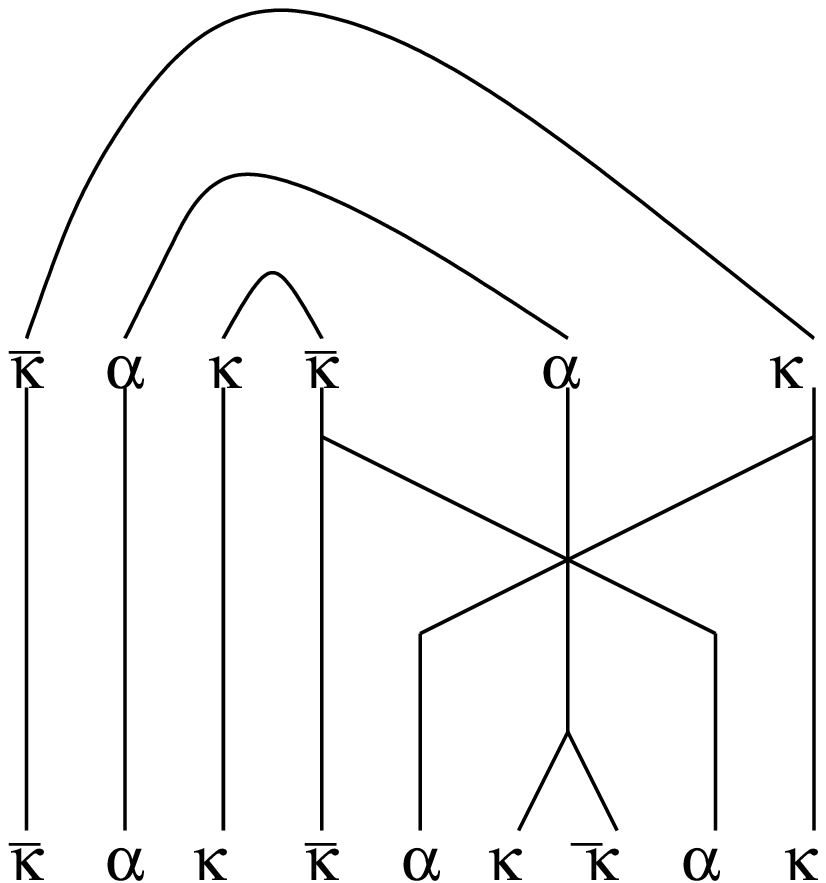} {1.2in} $. Its image under \hpic{iso} {0.8in} is $\frac{1}{\beta^{\frac{3}{2}}} \hpic{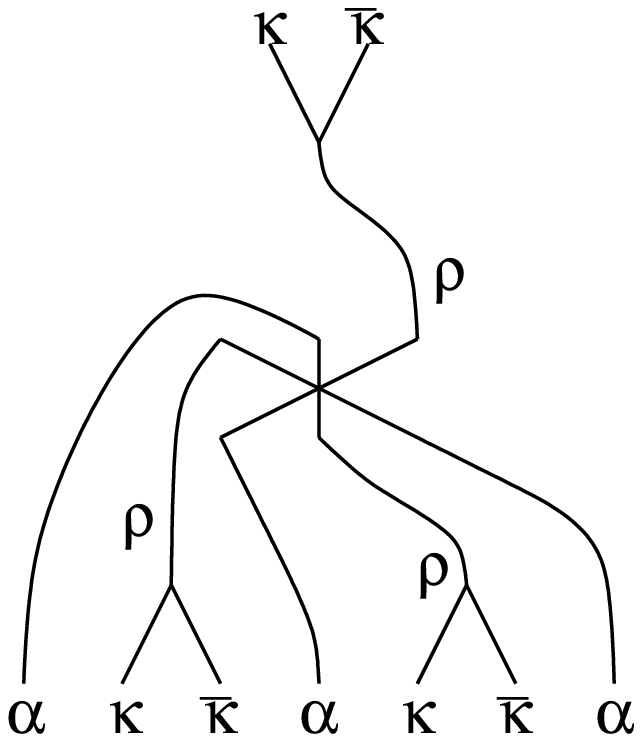} {1.2in} $. By \ref{eq3} these are equal.
\end{proof}

\begin{lemma}\label{l3}
We have $\displaystyle \frac{\beta}{\beta_1^2} (R \otimes Id_{\sigma}- Id_{\sigma}
\otimes R) = (Id_{\sigma} \otimes S) \circ S - (S \otimes Id_{\sigma}) \circ S$.
\end{lemma}

\begin{proof}
We have \\
 $ \frac{\beta}{\beta_1^2}( R \otimes Id_{\bar{\kappa } \alpha \kappa}- Id_{\bar{\kappa} \alpha \kappa} \otimes R) = \frac{1}{\beta_1^2}(\hpic{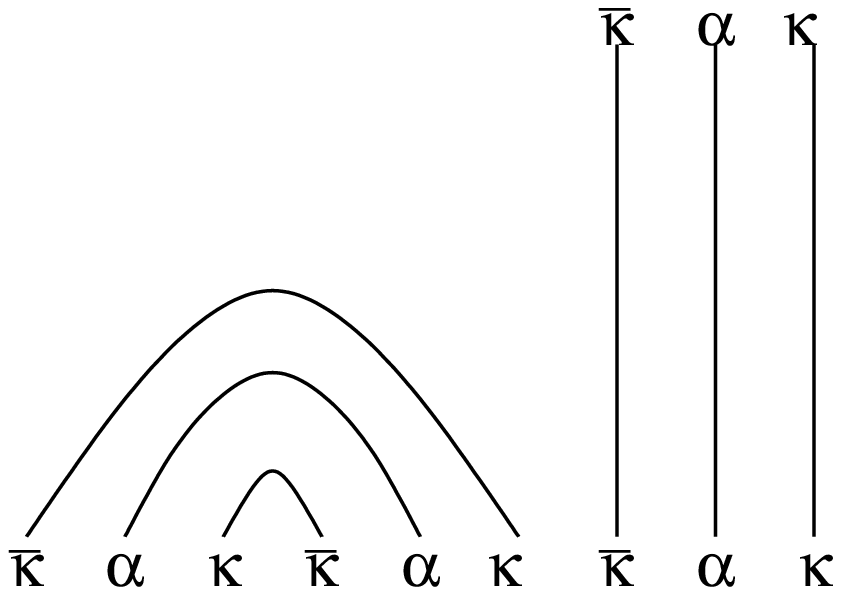} {1.0in} - \hpic{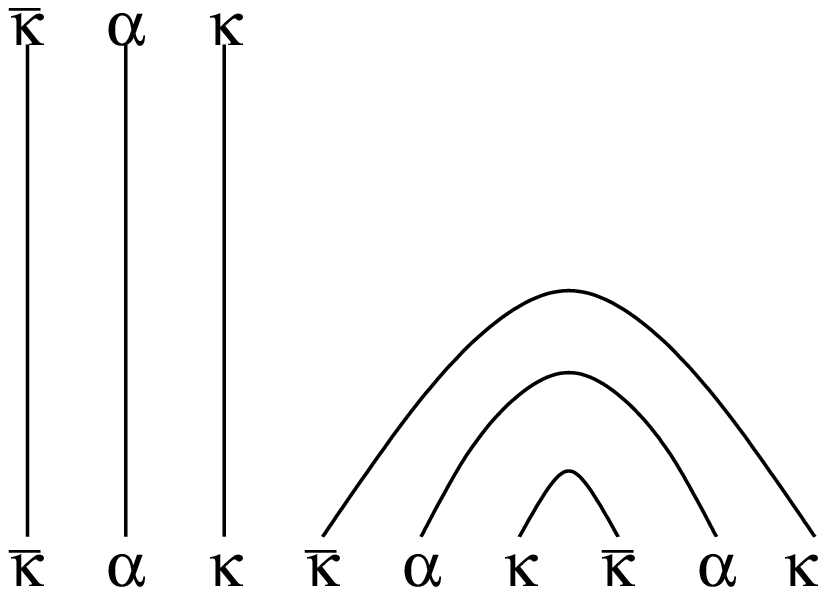} {1.0in} )$.  \\
 Its image under \hpic{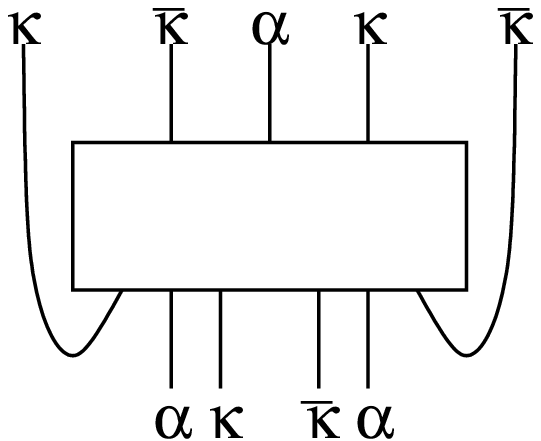} {0.8in} is $\frac{1}{\beta_1^2}  (\hpic{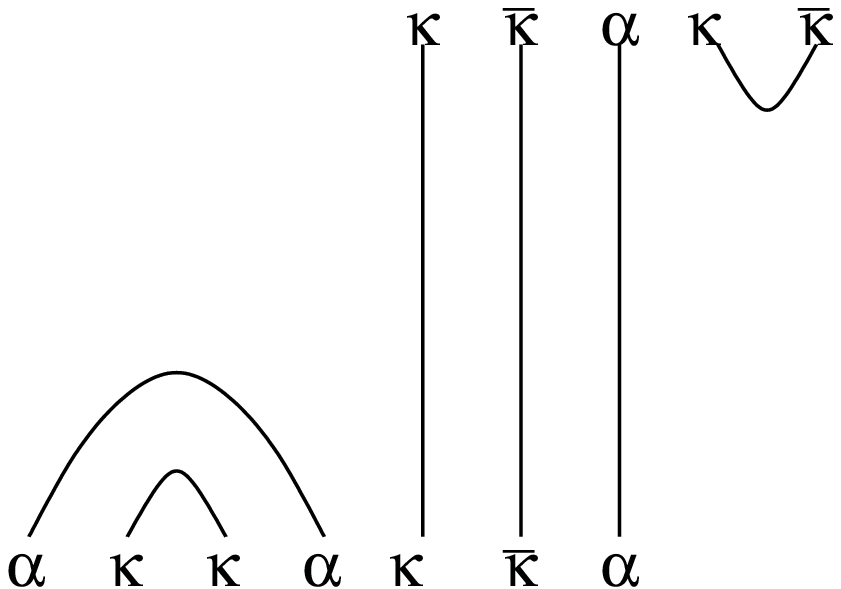} {1.0in} - \hpic{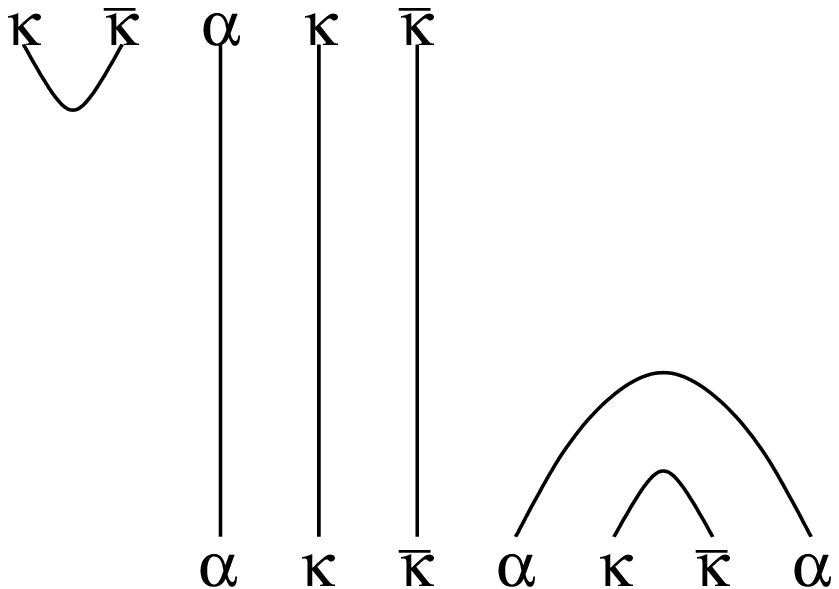} {1.0in} )$. Using $$\hpic{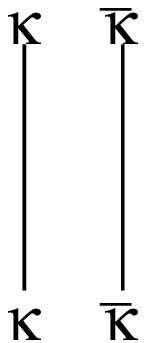} {0.8in} = \frac{1}{\beta} \hpic{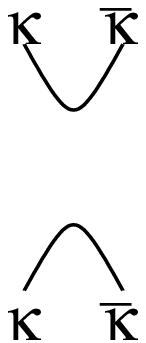} {0.8in} + \frac{\beta_1}{\beta} \hpic{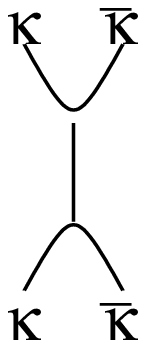} {0.8in} $$, we get $\frac{1}{\beta_1^2} ([\frac{1}{\beta}\hpic{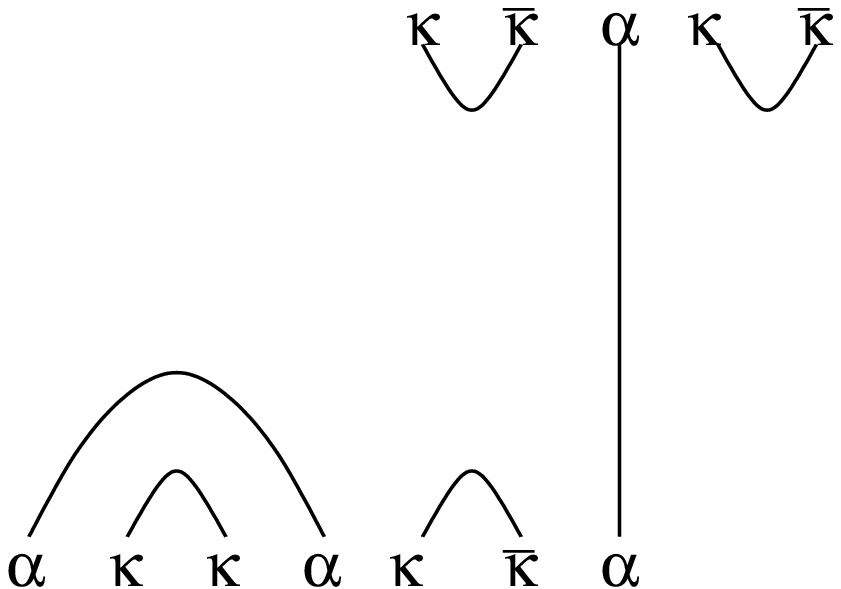} {1.0in} + \frac{\beta_1}{\beta}\hpic{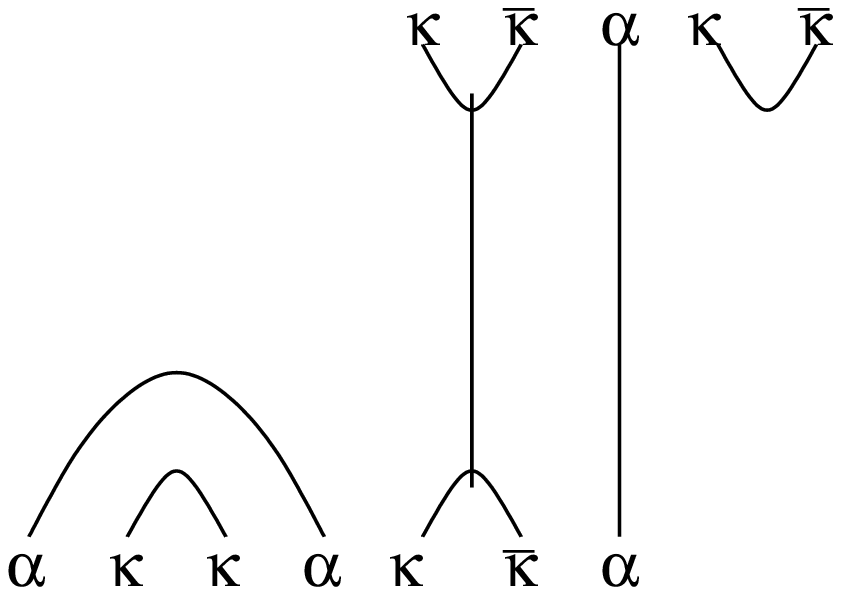} {1.0in} ]\\- [\frac{1}{\beta}\hpic{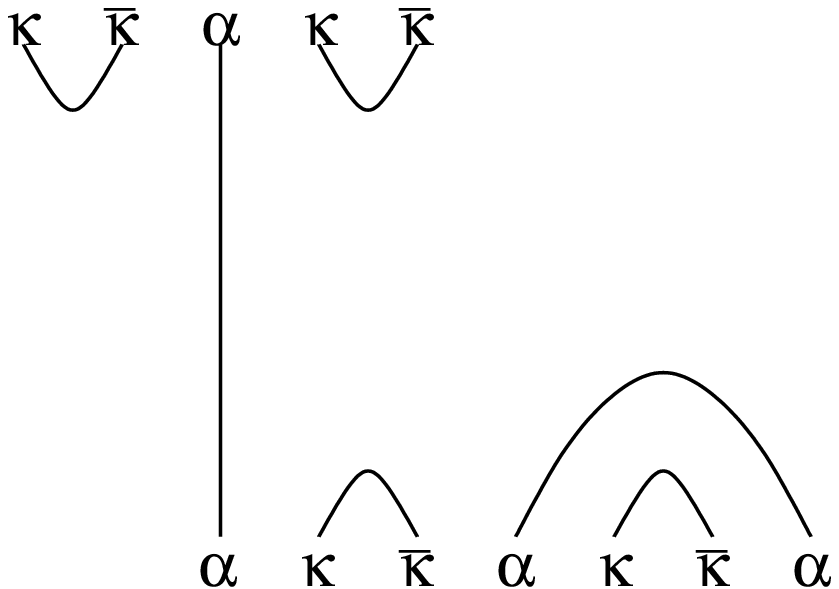} {1.0in} + \frac{\beta_1}{\beta}\hpic{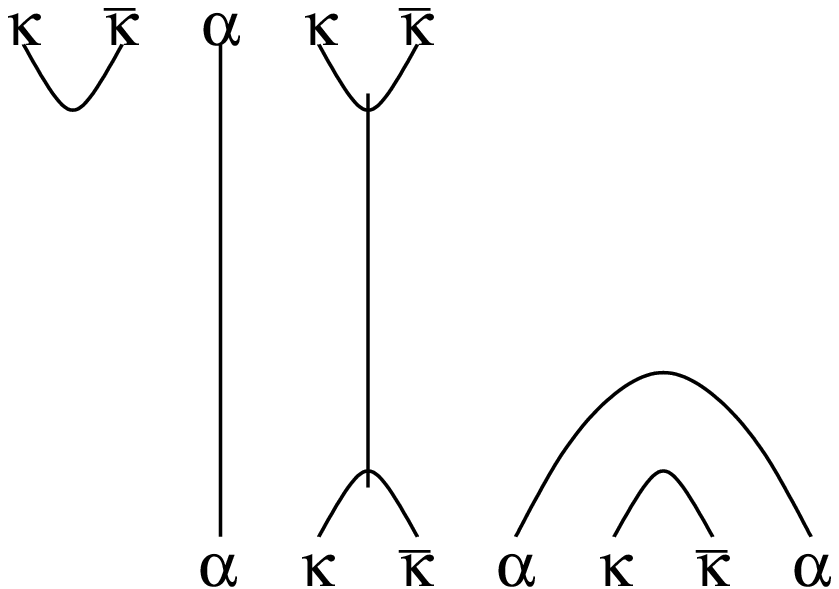} {1.0in} ] )\\=\frac{1}{\beta \beta_1} (\hpic{Rlt2} {1.0in} - \hpic{Rrt2} {1.0in} )$ by \ref{latelemma}. 

On the other hand,  \\
 $ (S \otimes Id_{\bar{\kappa} \alpha \kappa}) \circ S - (Id_{\bar{\kappa} \alpha \kappa} \otimes S) \circ S = \frac{1}{\beta} (\hpic{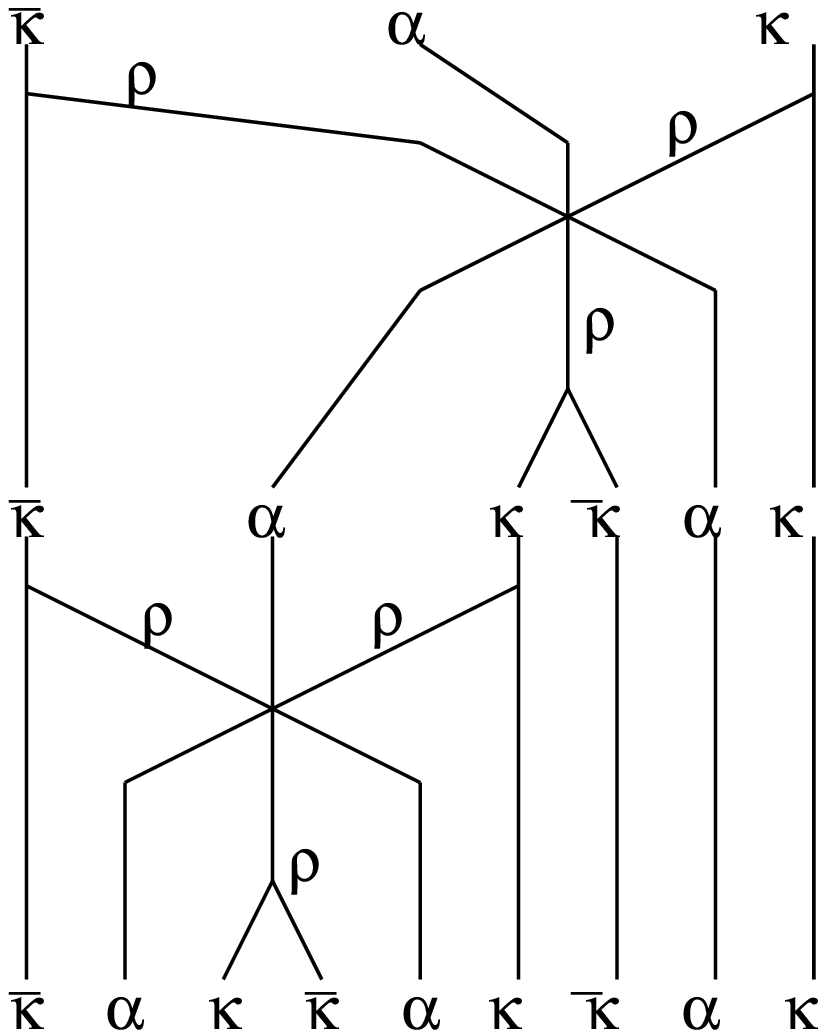} {1.4in} - \hpic{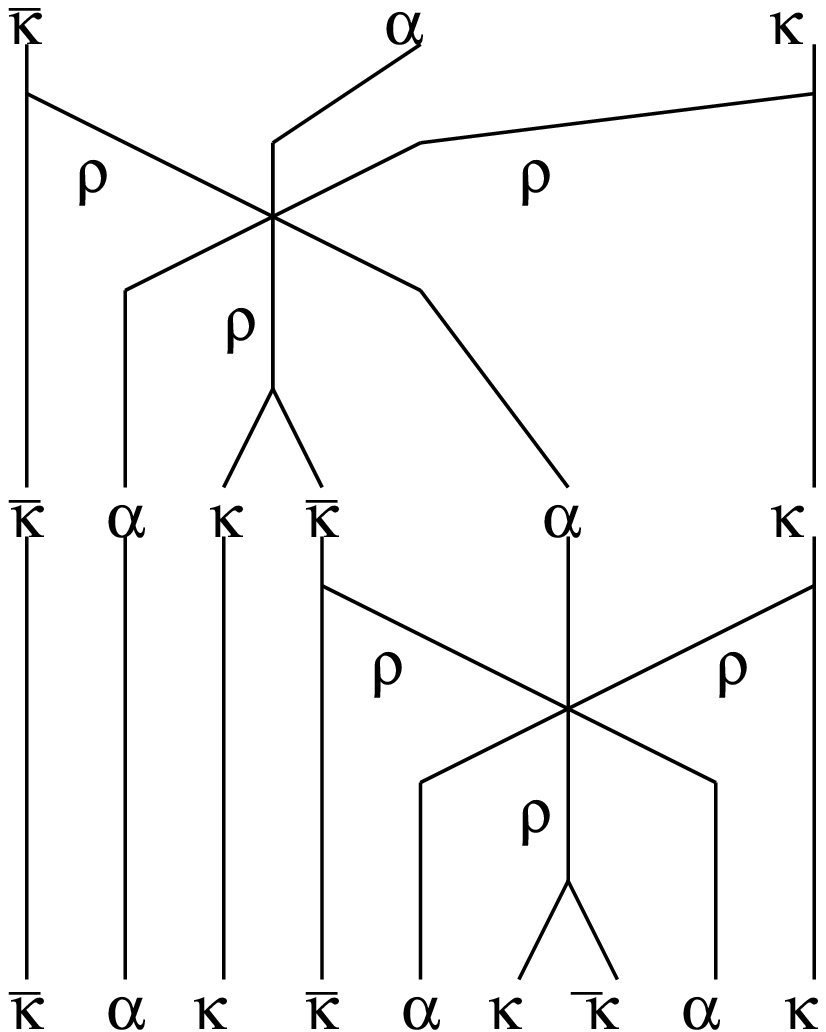} {1.4in} $). Its image under \hpic{iso2} {0.8in} is \\
 $\frac{1}{\beta}(\hpic{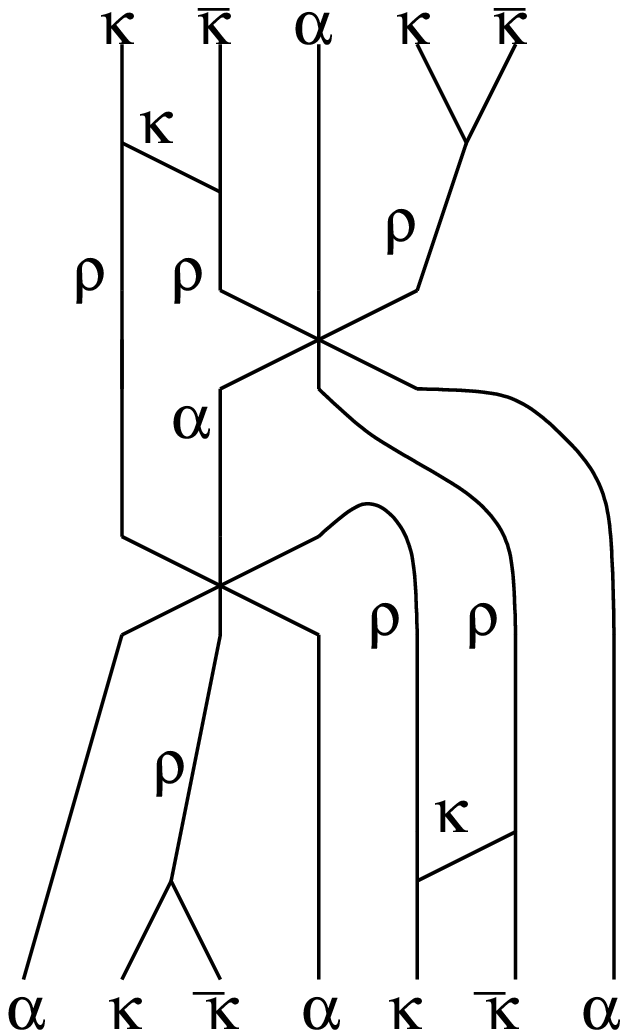} {1.4in} - \hpic{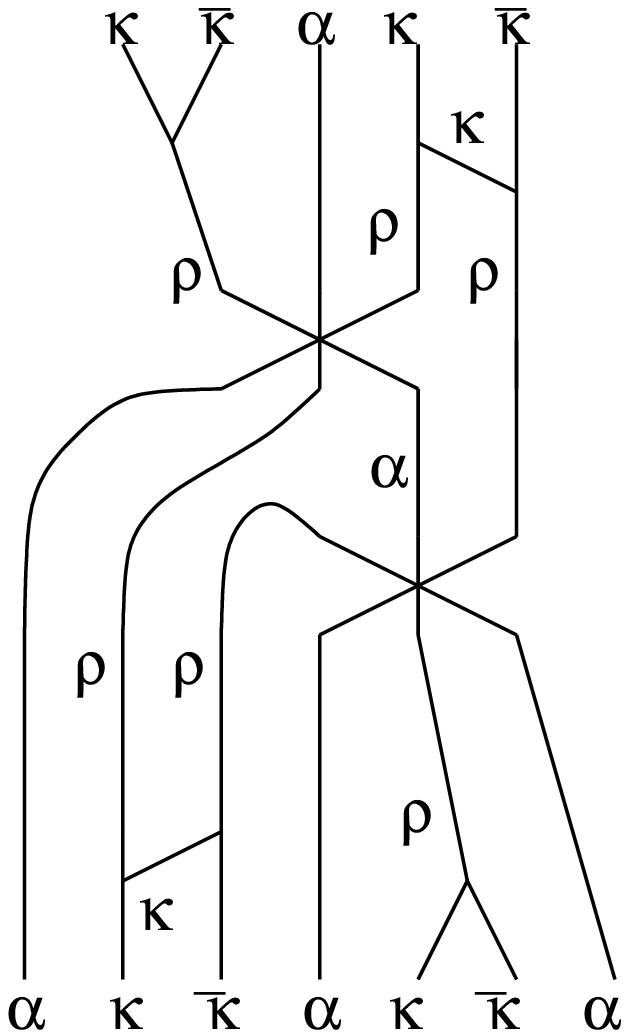} {1.4in} ) \\ = \frac{1}{\beta}[(\frac{1}{\beta_1^2} \hpic{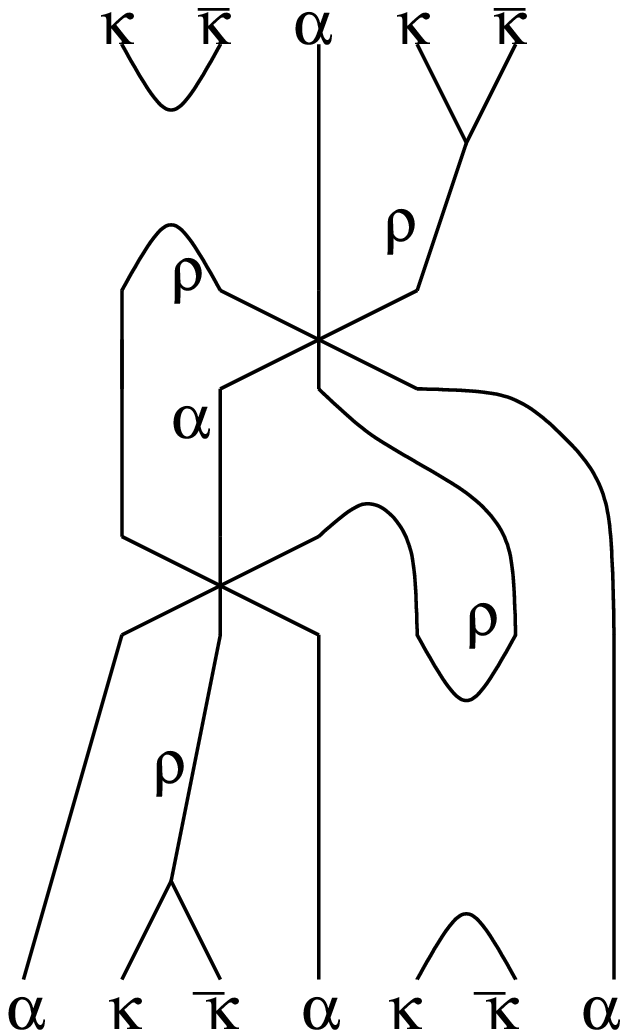} {1.4in} + \frac{\beta_2}{\beta_1^2} \hpic{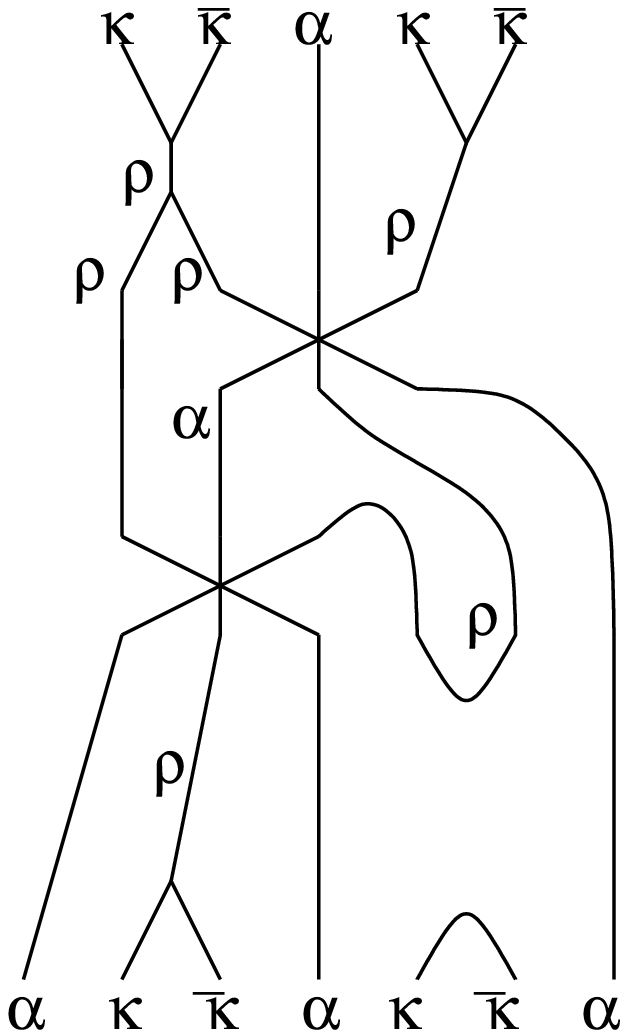} {1.4in} + \frac{\beta_2^2}{\beta_1^2} \hpic{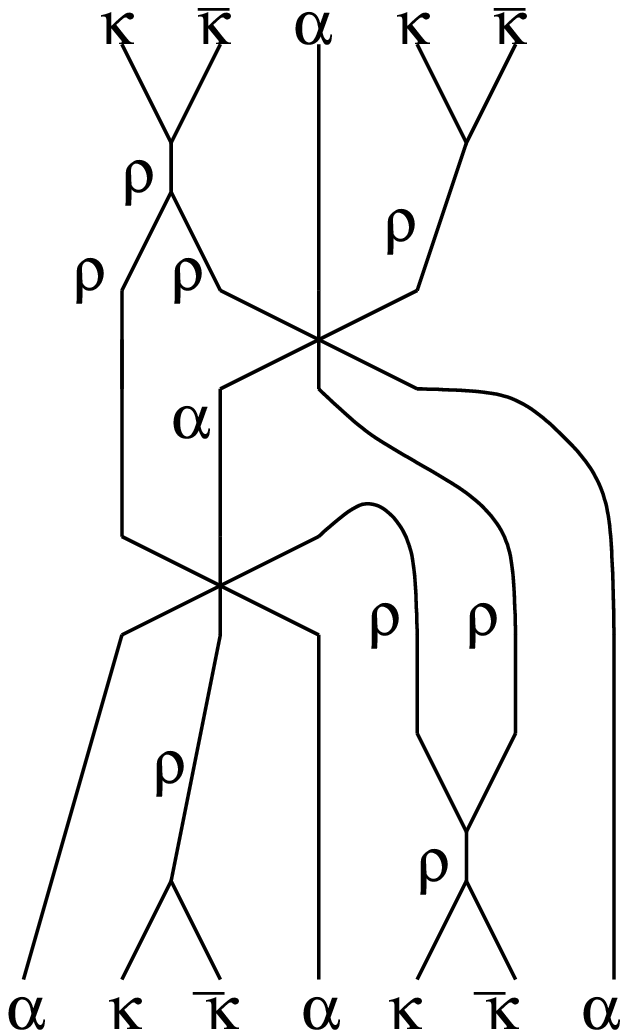} {1.4in} + \frac{\beta_2}{\beta_1^2} \hpic{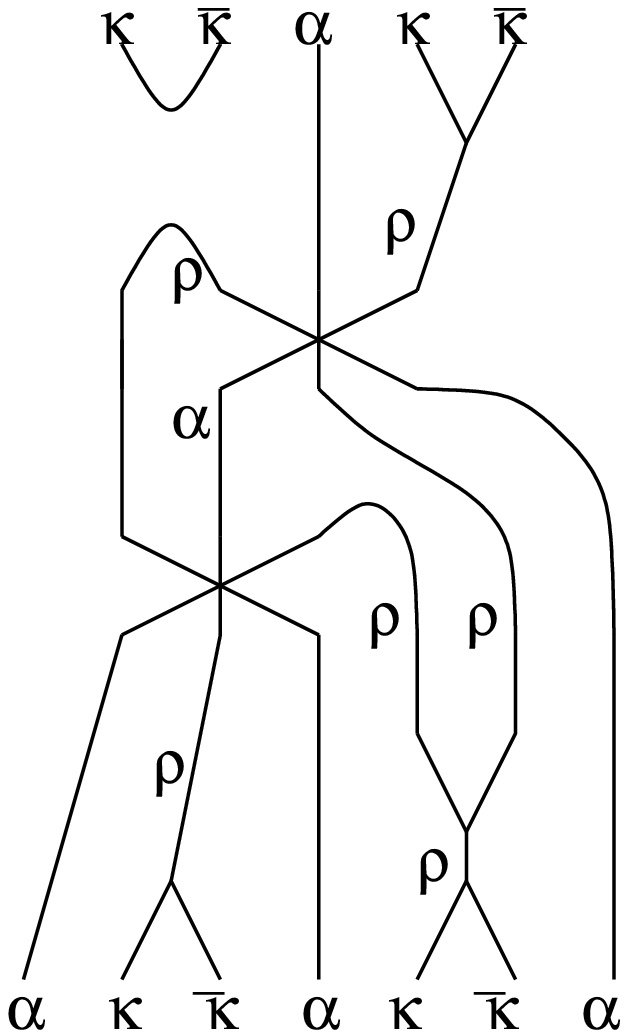} {1.4in} )\\ - (\frac{1}{\beta_1^2} \hpic{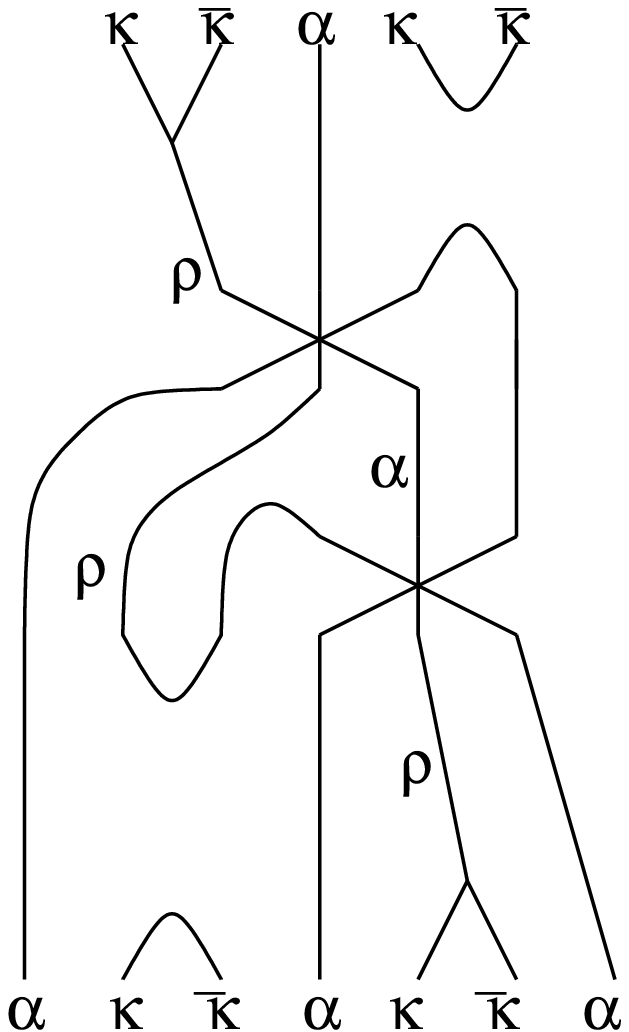} {1.4in} + \frac{\beta_2}{\beta_1^2} \hpic{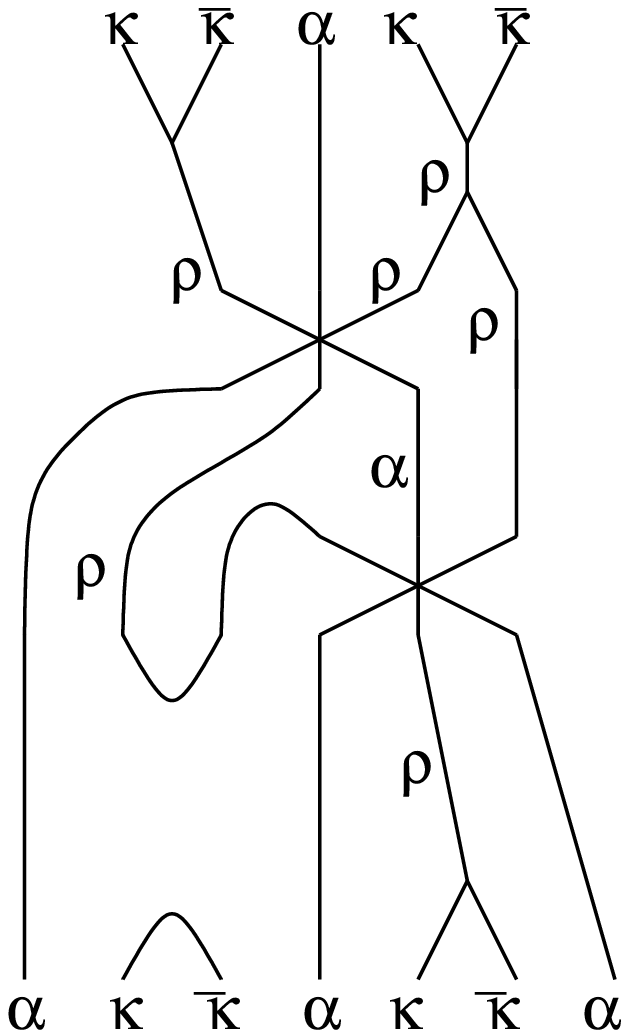} {1.4in} + \frac{\beta_2^2}{\beta_1^2} \hpic{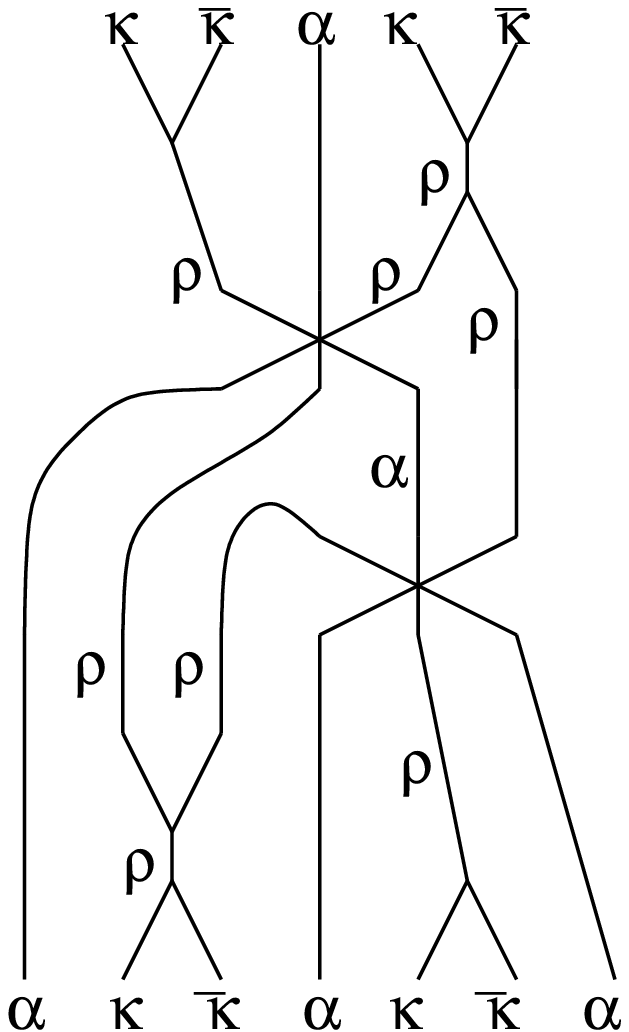} {1.4in} + \frac{\beta_2}{\beta_1^2} \hpic{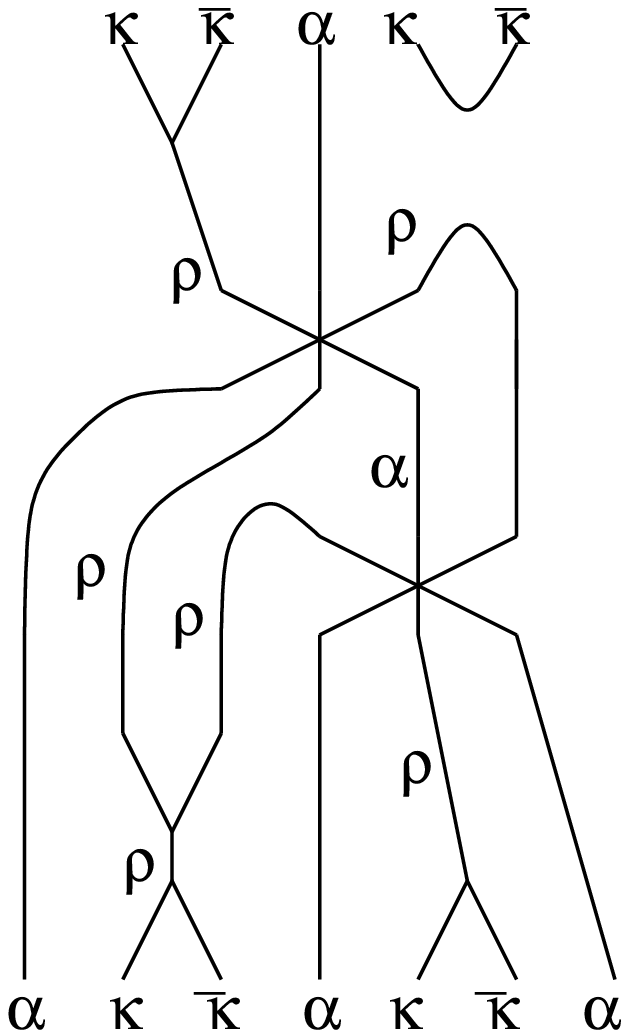} {1.4in} )] $, where we have used \ref{skein} twice on each side. Since $(\alpha \rho, \rho \alpha \rho )=(\rho \alpha, \rho \alpha \rho )=(\alpha \rho, \alpha \rho \alpha \rho \alpha )=(\rho \alpha, \alpha \rho \alpha \rho \alpha ) = 0$, the second and fourth terms in each summand are $0$. Moreover, by \ref{eq2} the third terms are the same so they cancel. That leaves  \\
 $\frac{1}{\beta \beta_1^2} (\hpic{eq2lhsafter_3} {1.4in} -\hpic{eq2rhsafter_3} {1.4in} ) \\
 =\frac{1}{\beta \beta_1} ( \hpic{Rrt2} {1.0in} - \hpic{Rlt2} {1.0in} )$, \\ 
 where we have used \ref{eq1}.
\end{proof}

\begin{theorem} \label{qthm}
 The bimodule $\gamma = Id_N  \oplus \bar{\kappa} \alpha \kappa $ admits a
Q-system, which is unique up to equivalence.
\end{theorem}
\begin{proof}
Existence is immediate from \ref{l1}, \ref{l2}, \ref{l3} and \ref{qsystem_1plus}. For uniqueness, note that since $dim(\kappa \alpha \bar{\kappa},\kappa \alpha \bar{\kappa} \kappa \alpha \bar{\kappa} )=1 $, $S$ is determined up to a scalar. For the equation in \ref{l3} to hold, that scalar is determined up to a sign, which means the Q-system is determined up to equivalence (see \cite{GI}, Lemma 3.5). 
\end{proof}

Once existence of the Q-system is known, the principal graph of the corresponding subfactor can be easily computed from the Asaeda-Haagerup fusion rules:

$$\hpic{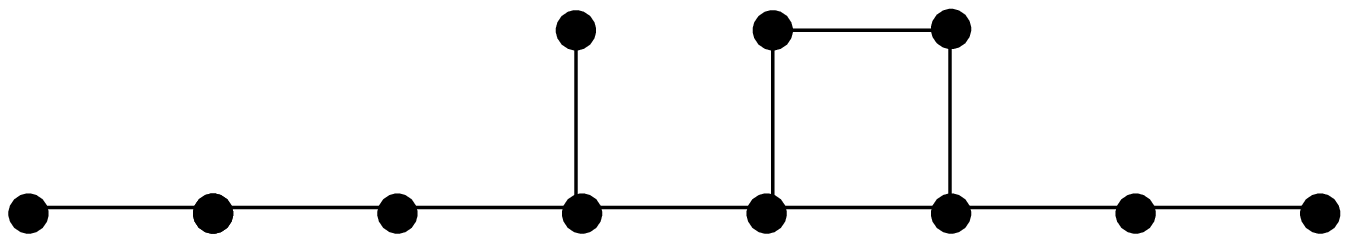} {0.4in} $$

The dual graph was computed using the subfactor atlas \\(http://tqft.net/wiki/Atlas\_of\_subfactors) and sent to the authors by Noah Snyder. It is:

$$\hpic{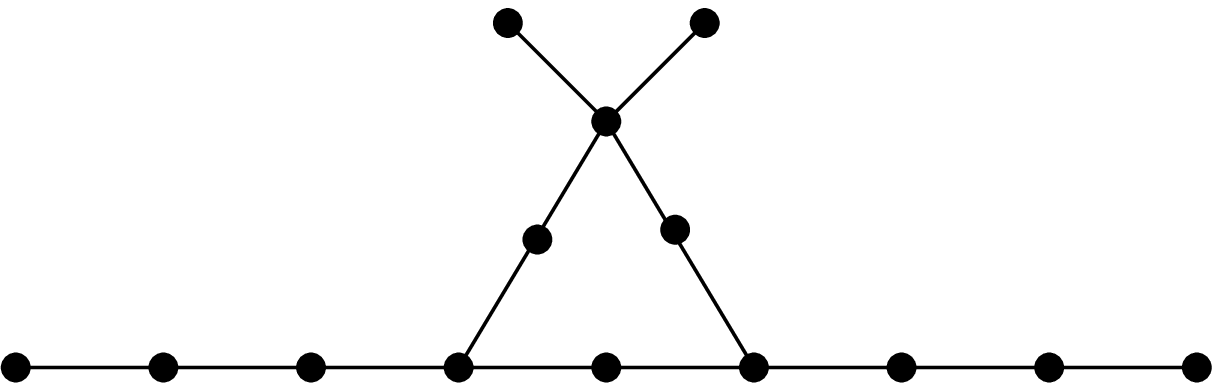} {0.8in} $$

Note that the dual graph possesses an order two symmetry very similiar to that of the original graph. We therefore conjecture that our construction may be iterated once more to obtain a Q-system in the Asaeda-Haagerup category with index $\frac{9+\sqrt{17}}{2} $. Checking this should be a straightforward computation, but we would first need some data from an analogue of Asaeda and Haagerup's original computation, applied to the new ``AH+1'' subfactor.

\begin{theorem}
 There is an irreducible, noncommuting but cocommuting quadrilateral whose upper sides are the Asaeda-Haagerup subfactor. Such a quadrilateral is unique up to isomorphism of the planar algebra.
\end{theorem}
\begin{proof}
 Let $P \subset M $ be an Asaeda-Haagerup subfactor with $\kappa = {}_P M_M$ corresponding to the fundamental vertex on the Asaeda-Haagerup principal graph. By \ref{qthm}, we can find a subfactor $N \subset P $ such that $\iota\bar{\iota} \cong Id_P \oplus \bar{\kappa} \alpha \kappa$, where $\iota = {}_N P_P $. Then ${}_N M_M = \iota \kappa$, and $dim(\iota \kappa, \iota \kappa )=dim(\iota \bar{\iota}, \bar{\kappa} \kappa ) = \dim(Id_P \oplus \bar{\kappa} \alpha \kappa, \bar{\kappa} \kappa )=1 $, so $N \subset M$ is irreducible. 

By a slight abuse of notation, we will let $\alpha$ denote the dimension $1$ $M-M$ bimodule and also the corresponding ouer auotmorphim of $M$. We have  $dim(\bar{\kappa} \bar{\iota}\iota \kappa, \alpha) = dim( \bar{\iota}\iota, \bar{\kappa} \alpha \kappa)$, so ${}_M M_1 {}_M $ contains a copy of $[\alpha]$, where $N \subset M \subset M_1 $ is the basic construction.  Take a representative of $[\alpha]$ in the Galois group of $N \subset M$ and then set $Q=\alpha(P)$, and consider the quadrilateral $\begin{array}{ccc}
P&\subset &M \cr
\cup& &\cup \cr
N&\subset &Q
\end{array}$. Then $P \neq Q$, or else either $N \subset P$ or $P \subset M$ would have to have a nontrivial Galois group, which is not the case. Since ${}_N P_N \cong {}_N \alpha(P)_N = {}_N Q_N$, the quadrilateral does not commute (see \cite{GI}, Theorem 3.10). On the other hand, if $P \subset M \subset \bar{P}$ and $ Q \subset M \subset \bar{Q} $ are each the basic construction, then ${}_M L^2(\bar{P})_M \cong \kappa \bar{\kappa} \cong Id_M \oplus \rho \ncong Id_M \oplus \alpha \rho \alpha \cong \alpha \kappa \bar{\kappa} \alpha \cong {}_M L^2(\bar{Q})_M$. By \cite{GrJ}, Lemma 4.2.1, the quadrilateral cocommutes. 
Uniqueness follows from \cite{GI}, Theorem 4.8.
\end{proof}
This thoerem answers the conjecture in \cite{GI}, Remark 5.16.
\begin{remark}
If the iterated equations hold and the Q-system for index $\frac{9+\sqrt{17}}{2}$ exists, there would similarly be a noncommuting quadrilateral whose upper sides are the ``AH+1'' subfactor and whose lower sides  have index $\frac{9+\sqrt{17}}{2} $.
\end{remark}

\end{document}